\documentclass{amsart}
\usepackage{mathtools}
\usepackage{amsmath,amssymb,amsthm,color,mathrsfs,array}
\usepackage{multirow}
\usepackage{verbatim}
\usepackage{url}

\newcommand{\uvec}[1]{\mathit{uvec}(#1)}

\newcommand{\tr}[1]{\operatorname{tr}(#1)}

\newcommand{\re}{\mathbb{R}}
\newcommand{\mR}{\mathbb{R}}

\newcommand{\mC}{\mathbb{C}}

\newcommand{\cpx}{\mathbb{C}}

\newcommand{\N}{\mathbb{N}}

\def\af{\alpha}

\def\rank{\mbox{rank}}

\newcommand{\reff}[1]{(\ref{#1})}

\renewcommand{\vec}[1]{\mathit{vec}(#1)}
\newcommand{\qmod}[1]{\mbox{QM}[#1]}
\newcommand{\ideal}[1]{\mbox{Ideal}[#1]}

\newcommand{\st}{\mathit{s.t.}}

\newcommand{\bdes}{\begin{description}}
	\newcommand{\edes}{\end{description}}
\newcommand{\bal}{\begin{align}}
	\newcommand{\eal}{\end{align}}
\newcommand{\bnum}{\begin{enumerate}}
	\newcommand{\enum}{\end{enumerate}}
\newcommand{\bit}{\begin{itemize}}
	\newcommand{\eit}{\end{itemize}}
\newcommand{\bea}{\begin{eqnarray}}
	\newcommand{\eea}{\end{eqnarray}}
\newcommand{\be}{\begin{equation}}
	\newcommand{\ee}{\end{equation}}
\newcommand{\baray}{\begin{array}}
	\newcommand{\earay}{\end{array}}
\newcommand{\bsry}{\begin{subarray}}
	\newcommand{\esry}{\end{subarray}}
\newcommand{\bca}{\begin{cases}}
	\newcommand{\eca}{\end{cases}}
\newcommand{\bcen}{\begin{center}}
	\newcommand{\ecen}{\end{center}}
\newcommand{\bbm}{\begin{bmatrix}}
	\newcommand{\ebm}{\end{bmatrix}}
\newcommand{\bmx}{\begin{matrix}}
	\newcommand{\emx}{\end{matrix}}
\newcommand{\bpm}{\begin{pmatrix}}
	\newcommand{\epm}{\end{pmatrix}}
\newcommand{\btab}{\begin{tabular}}
	\newcommand{\etab}{\end{tabular}}

\theoremstyle{plain}
\newtheorem{theorem}{Theorem}[section]

\newtheorem{prop}[theorem]{Proposition}

\newtheorem{lemma}[theorem]{Lemma}

\newtheorem*{claim*}{Claim}
\newtheorem{thm}[theorem]{Theorem}
\theoremstyle{definition}

\newtheorem{exm}[theorem]{Example}
\newtheorem{alg}[theorem]{Algorithm}

\newtheorem{remark}[theorem]{Remark}

\numberwithin{equation}{section}
\numberwithin{table}{section}

\begin{document}

\title[Testing copositivity over the semidefinite cone]
{A  finite-termination  algorithm for testing copositivity over the semidefinite cone}

\author[Lei Huang]{Lei Huang}
\address{Lei Huang,Department of Mathematics,
	University of California San Diego,
	9500 Gilman Drive, La Jolla, CA, USA, 92093.}
\email{leh010@ucsd.edu}

\author[Lingling Xie]{Lingling Xie}
\address{Lingling Xie,  School of Mathematics and Statistics, Hefei Normal University,
No. 327 Jinzhai Road, Hefei, Anhui, China, 230061.}
\email{linglingxie@amss.ac.cn}



\maketitle

\begin{abstract}	

This paper proposes an efficient algorithm  for testing copositivity of homogeneous polynomials over the positive semidefinite cone. The algorithm is based on a novel matrix optimization reformulation and requires solving a hierarchy of semidefinite programs.  Notably, it always terminates in finitely many iterations. If a homogeneous polynomial is copositive over the positive semidefinite cone, the algorithm provides a certificate; otherwise, it returns a vector that refutes copositivity. Building on a similar idea, we further propose an algorithm to test copositivity  over the direct product of the positive semidefinite cone and the nonnegative orthant.
 Preliminary numerical experiments
demonstrate the effectiveness of the proposed methods.

\end{abstract}

\section{Introduction}

Let $K\subseteq \mR^n$ be a nonempty   closed convex cone. An $n$-by-$n$ real symmetric matrix $A$ is said to be  {\it $K$-copositive} or {\it copositive over  $K$}   if 
the associated quadratic polynomial $x^TAx\geq 0$ for all $x\in K$. The set of all $K$-copositive matrices  forms a closed convex cone,  denoted by $\mathcal{COP}^n(K)$, which is also referred to as the set-semidefinite cone in some literature \cite{ank21,egpj13}.  A matrix $B$ is called $K$-completely positive if  it can be expressed as $B=\sum\limits_{i=1}^t u_i^Tu_i$ for some vectors $u_i\in K$ $(i=1,\dots,t)$.  The set of all  $K$-completely positive matrices is denoted by $\mathcal{CP}^n(K)$, and its dual cone is $\mathcal{COP}^n(K)$. 
When $K$ is the nonnegative orthant  (i.e., $K=\mR_{+}^n$), the cones $\mathcal{COP}^n(K)$ and $\mathcal{CP}^n(K)$ reduce to  the standard copositive cone and completely positive cone, respectively. For an introduction to the basic theory of copositive matrices, we refer to \cite{boi12,bu12,dum10,hus10}.

Linear
 optimization over   $\mathcal{COP}^n(K)$ or $\mathcal{CP}^n(K)$, known as    copositive programming, has broad applications, as many NP-hard problems can be  formulated as copositive programs. In particular, there is high interest when $K$ is the nonnegative orthant, the second order cone, the positive semidefinite cone or their direct products. For instance, 
many nonconvex quadratic programs can be reformulated as  copositive programs over the nonnegative orthant  \cite{bu09},  the direct product of   a nonnegative orthant and second-order cones  \cite{bh12}, the direct product of   a nonnegative orthant and   a positive semidefinite cone \cite{bh12}, or  general cones \cite{egpj13,stz03}.   In \cite{bmp16},  a rank-constrained semidefinite programming 
problem is formulated as a copositive program over  the direct product of a nonnegative orthant and three positive semidefinite cones. Xu and  Hanasusanto  \cite{xg24}  derive copositive programming reformulations over the general cones  in the context of popular linear decision rules for generic  multistage robust optimization  problems. In \cite{bg23}, applications of copositive
programming in robust optimization and stochastic optimization are discussed.  For more applications in quadratic programming, polynomial optimization, graph theory and optimization under uncertainty and risk, we  refer to \cite{bg21,vdng05,gmak23,kkt20,stz03,xb18}.

A fundamental problem in   copositive programming is to test whether a given matrix $A$ is $K$-copositive for a given closed convex cone $K$. When $K$ is the nonnegative orthant $\mR^n_+$, it is well-known that this recognition problem is  co-NP complete \cite{dpgl14,mkks87}. Many efficient methods have been proposed for testing  $\mR^n_+$-copositive matrices,    based on linear programming \cite{vdng07,dp19},  mixed-integer linear programming \cite{ansk21}, simplicial partition \cite{bsdm08,sbd12}, difference of convex programming \cite{boe13,dmhj13}, polynomial optimization \cite{nyz18},  branch-and-bound algorithms \cite{buvd08,vvdng05}.  When $K$ is the second order cone, Loewy and Schneider \cite{loe75} showed that $A$ is  $K$-copositive if and only if there exists  $\mu\geq 0$ such that the matrix $A-\mu J$ is positive semidefinite, where $J$ is the diagonal matrix with entries $-1,\dots,-1,1$ on the diagonal. Thus, copositivity over the second-order cone can be detected by solving the semidefinite feasibility problem. When $K$ is the positive semidefinite cone $\mathcal{S}_+^{n}$, a gradient projection method  was introduced in 
\cite{fgnr24} to  test $\mathcal{S}_+^{n}$-copositivity. This algorithm can certify that  a matrix $A$ is not copositive if the output is negative. However, it is hard to certify that $A \in \mathcal{COP}^n(K)$. For a general closed convex cone $K$,  approximation hierarchies for  $\mathcal{COP}^n(K)$  exist, which are   sequences of
cones  described  by linear 
inequalities or linear matrix
inequalities;   see \cite{dpj13,ank21,egpj13,las14,mnn24a,mnn24b,pap00,VLau15,zvp06}. For most of these methods, if $A$  lies in
the interior of  $\mathcal{COP}^n(K)$,  $K$-copositivity can be detected. However,  it is more challenging to obtain a certificate when $A$  lies  on the boundary of $\mathcal{COP}^n(K)$ or  to provide a refutation when  $A$ does not belong to  $\mathcal{COP}^n(K)$.

\subsection*{Contributions}
This paper  studies the problem of testing $\mathcal{S}^n_+$-copositivity  of homogeneous   polynomials.  
 Let   
\[
x:=(x_{11},x_{12},\dots,x_{1n},x_{22},x_{23},\dots,x_{n-1,n},x_{nn})
\]
be  the vector consisting of $\sigma(n):=\frac{1}{2}n(n+1)$ variables, and let $X(x)$
be the $n$-by-$n$ symmetric  matrix whose $(i,j)$-th entry is  $x_{ij}$ for $i\leq j$. A homogeneous polynomial $f(x)$ is  said to be  {\it $\mathcal{S}^n_+$-copositive}  if $f(x)\geq 0$ for every $x\in \mR^{\sigma(n)}$ satisfying $X(x)\succeq 0$. Note that a  symmetric matrix $A\in \mR^{\sigma(n)\times \sigma(n)}$ is $\mathcal{S}^n_+$-copositive if and only if the quadratic homogeneous polynomial $x^TAx$ is $\mathcal{S}^n_+$-copositive. 

Consider the  polynomial matrix optimization problem
 \be  \label{int:copo}
\left\{ \baray{rl}
\min & f(x)  \\
\st &\operatorname{Tr}(X(x))=1,\\
& 	X(x) \succeq 0, \\
\earay \right.
\ee
where the inequality $X(x) \succeq 0$ means that 
$X(x)$ is positive semidefinite. 
It is clear that $f(x)$ is   $\mathcal{S}^n_+$-copositive if and only if the optimal value of  \reff{int:copo} is nonnegative. The matrix  Moment--SOS  hierarchy  proposed in \cite{HDLJ,schhol}, which consists of a hierarchy of nested semidefinite relaxations, can be  applied to solve  \reff{int:copo} globally.   However, this hierarchy only guarantees asymptotic convergence and may fail to have finite convergence (see \cite{hn24}). Moreover,  it is often difficult to certify  finite convergence in practice even when  it does occurs, as the commonly used flat truncation condition \cite{HDLJ,Lau09,niebook} ( a key criterion for detecting finite convergence) may not be satisfied.

In this paper, we propose an efficient algorithm with finite  termination to test $\mathcal{S}^n_+$-copositivity. Our major results are:

\bit

\item We investigate the first order optimality conditions of  \reff{int:copo} and derive explicit expressions for the  Lagrange multipliers in terms of the decision variable $x$. Then, we propose a strengthened reformulation of \reff{int:copo}. This is inspired by the recent work \cite{hnlme}.

\item By utilizing the matrix Moment-SOS relaxations to solve the strengthened reformulation, we propose Algorithm \ref{ag:CopoSDP}   to test $\mathcal{S}^n_+$-copositivity of homogeneous polynomials, which involves solving a sequence of semidefinite programs.  Interestingly, the algorithm always terminates in finitely many iterations. If a homogeneous polynomial $f(x)$ is $\mathcal{S}^n_+$-copositive, Algorithm \ref{ag:CopoSDP}  provides a certificate;  otherwise, it returns a vector $u\in \mR^{\sigma(n)}$ such that $	X(u) \succeq 0$ and $f(u)<0$, thereby refuting  $\mathcal{S}^n_+$-copositivity. To the best of the authors' knowledge, this is the first algorithm that can test $\mathcal{S}^n_+$-copositivity for all homogeneous polynomials, with a guarantee of finite termination.

\item    Following a similar idea,    we  propose an algorithm to test copositivity of homogeneous polynomials over  the direct product of the positive semidefinite cone and the nonnegative orthant (see Algorithm \ref{ag2:CopoSDP}). It also has finite termination.

\eit

\vspace{0.3cm}
  
  	The paper is organized as follows.
  Section~\ref{sc:pre} reviews  optimality conditions for nonlinear semidefinite optimization
  and some basics in  polynomial matrix optimization.  Section~\ref{sec:refor} constructs the  strengthened reformulation of \reff{int:copo} and study   basic properties  of its matrix Moment-SOS relaxations.
   Section~\ref{sec:alo} presents the algorithm and proves its finite termination. Section \ref{sc:othsem}  presents an algorithm to test copositivity over the direct product of the positive semidefinite cone and the nonnegative orthant. Section~\ref{sec:num} reports some preliminary numerical experiments.
  We  make some conclusions and discussions in Section~\ref{sec:dis}.

\section{Preliminaries}
\label{sc:pre}

	\subsection*{Notation}\label{sec:notation}
The symbol $\mathbb{N}$ (resp., $\mathbb{R}$, $\cpx$) denotes the set of
nonnegative integers (resp., real numbers, complex numbers). The notation $e_i$ denotes the   
$i$th standard basis vector. For a real number $t,\lceil t\rceil$ represents the smallest integer  greater than or equal to $t$. For a matrix $A$, $A^{T}$ denotes its transpose, $\|A\|_F$ represents its Frobenius norm and $\operatorname{tr}(A)$ denotes its trace. 
Let $\mathbb{R}[x]:=\mathbb{R}\left[x_1, \ldots, x_n\right]$ (resp., $\mathbb{C}[x]$)
be the ring of  polynomials in $x:=\left(x_1, \ldots, x_n\right)$ with real coefficients (resp., complex coefficients), and let $\mathbb{R}[x]_d$
be the set of polynomials with degrees at most $ d$. The notation  $ \operatorname{deg}(p)$ denotes the total degree of a polynomial $p$.
Let   $\mathcal{S}^{\ell}$ 
	(resp., $\mathcal{S}_{+}^{\ell}$) denote
the set of all $\ell$-by-$\ell$  real symmetric matrices
(resp.,  positive semidefinite matrices).
For $\af = (\af_1, \ldots, \af_n) \in \N^n$ and an integer $d>0$,
define
\[
|\af| \coloneqq \af_1 + \cdots + \af_n,\quad \sigma(d)  \, \coloneqq \,   d(d+1)/2, \quad
\mathbb{N}_d^n \, \coloneqq \,  \left\{\alpha \in \mathbb{N}^n \mid
|\af| \leq d\right\},
\]
and  denote by $[x]_d^T$ the vector of all monomials with degrees $\leq d$, ordered lexicographically:
\[
[x]_d^T :=\, [1,  x_1, x_2, \ldots, x_1^2, x_1x_2, \ldots,
x_1^d, x_1^{d-1}x_2, \ldots, x_n^d ].
\]
For a smooth function $f(x)$,
we denote its gradient by $\nabla f(x)$ and its partial derivative with respect to $x_i$  by $f_{x_i}$.


\subsection{Optimality conditions for nonlinear semidefinite optimization}
\label{seccq}

In this subsection, we review  optimality conditions for nonlinear semidefinite optimization; see \cite{shap,shuni,sunde,yaya} for more details.

For an $m\times m$  smooth symmetric matrix-valued function $G(x)= ( G_{st}(x))_{s,t=1,\ldots,m}$, the derivative of  $G$ at $x$ is the linear mapping
$\nabla G(x): \mathbb{R}^n \rightarrow \mathcal{S}^m$
such that
\[
d \, := \, \left(d_1, \ldots, d_n\right) \, \mapsto \,
\nabla G(x)[d] \, := \, \sum_{i=1}^n d_i \nabla_{x_i} G(x),
\]
where $\nabla_{x_i} G(x) $ denotes the partial derivative of $G$
with respect to $x_i$, i.e.,
\[
\nabla_{x_i} G(x) \, = \, \left(\frac{\partial G_{st}(x)}{\partial x_i}\right)_{s,t=1,\dots,m}.
\]
The adjoint of $\nabla G(x)$ is the linear mapping
$\nabla G(x)^*: \mathcal{S}^m \rightarrow \mathbb{R}^n$
such that 
\[
X \, \mapsto \,
\nabla G(x)^*[X] \, := \, \bbm \tr{\nabla_{x_1} G(x)X}
& \cdots  & \tr{\nabla_{x_n} G(x), X} \ebm^T.
\]
	For a point $u \in \mR^n$ with $G(u) \succeq 0$, let $E$ be a matrix whose column vectors form a basis of the kernel 
	of $G(u)$.
	The (Bouligand) tangent cone to  $\mathcal{S}_{+}^m$
	at $G(u)$ is
	\[
	T_{\mathcal{S}_{+}^m}(G(u))  = \left\{N \in \mathcal{S}^m: E^T N E \succeq 0\right\}.
	\]
	The lineality space at $u$ is
	\[
	\operatorname{lin}\left(T_{\mathcal{S}_{+}^m}(G(u))\right) \, =  \,
	\left\{N \in \mathcal{S}^m: E^T N E=0\right\},
	\]
	which is the largest  linear subspace contained in $T_{\mathcal{S}_{+}^m}(G(u))$. 
	
	Consider the  matrix
	optimization
	\be  \label{nsdp}
	\left\{ \baray{rl}
	\min & f(x)  \\
	\st & h_1(x)=0,\dots,h_{\ell}(x)=0,\\
	& 	G_1(x)\succeq 0,\dots,G_s(x)\succeq 0, \\
	\earay \right.
	\ee
	where $f(x),h_1(x),\dots,h_{\ell}(x)$ are  smooth functions, and each $G_t(x)$ is an $m_t\times m_t$  smooth symmetric  matrix-valued function.   
	Let $u$ be a feasible point of \reff{nsdp}.	Denote the Jacobian operator of $h:=(h_1,\dots,h_{\ell})$ at $u$ by $\mathcal{J}h(u)$,
	The {\it nondegeneracy condition} (NDC)
	is said to hold at $u$ if
	\be \label{CQ}
	\left[\begin{array}{c}
		\mathcal{J}h(u) \\
		\nabla G_1(u)\\
		\vdots\\
		\nabla G_s(u)
	\end{array}\right] \mathbb{R}^n+\left[\begin{array}{c}
		0 \\
		\operatorname{lin}\left(T_{\mathcal{S}_{+}^{m_1}}(G_1(u))\right)\\
		\vdots\\
		\operatorname{lin}\left(T_{\mathcal{S}_{+}^{m_s}}(G_s(u))\right)
	\end{array}\right]=\left[\begin{array}{l}
		\mathbb{R}^{\ell} \\
		\mathcal{S}^{m_1}\\
		\vdots\\
		\mathcal{S}^{m_s}
	\end{array}\right].
	\ee
 The following proposition provides an equivalent description for the  NDC.

	\begin{prop}[\cite{shuni,sunde}] \label{equindc}
		Suppose $u$ is a feasible point of \reff{nsdp} and  $\rank~ G_t(u)=r_t$ for $t=1,\dots,s$. Let   $\{q_1^{(t)},\dots,q_{m_t-r_t}^{(t)}\}$ be an arbitrary basis for the kernel of  $G_t(u)$. Then, the NDC \reff{CQ} holds at  $u$ if and only if the following vectors are linearly independent:
	{\small	\[
	\nabla h_1(u),\dots, \nabla h_{\ell}(u),	\nabla G_t(u)^*[\frac{q^{(t)}_i (q^{(t)}_j)^{T}+q^{(t)}_j (q^{(t)}_i)^{T}}{2}]~( 1 \leq i \leq j \leq m_t-r_t,~t=1,\dots,s).
	\]}
	\end{prop}

		The NDC is generally viewed as an analogue of the classical
	linear independence constraint qualification condition 
	in nonlinear programming, as it ensures the uniqueness of
	the Lagrange multipliers. Under the NDC, we can derive the first  order
	optimality conditions for  \reff{nsdp}.
 
	\begin{thm}[\cite{shap,sunde}] \label{KKT:cond}
		Suppose $u$ is a local minimizer of \reff{nsdp},
		and the  NDC \reff{CQ} holds at $u$.
		Then, there exist unique matrices $\Lambda_1\in \mathcal{S}^{m_1},\dots, \Lambda_s\in \mathcal{S}^{m_s}$,  and scalars $\mu_1,\dots,\mu_{\ell} $ such that
		\be \label{kkt}
		\begin{array}{r}
			\nabla f(u)-\sum\limits_{i=1}^{\ell} \mu_i \nabla h_{i}(u)-\sum\limits_{t=1}^{s}\nabla G_t(u)^*[\Lambda_t] = 0, \\
			 \Lambda_t \succeq 0, \,\,
			\tr{\Lambda_tG_t(u)}=0\,\,(t=1,\dots,s).
		\end{array}
		\ee
	\end{thm}
	
\vspace{.3cm}
			
\noindent  Since the matrices $\Lambda_t, G_t(u)$ are symmetric and positive semidefinite, the equation $\tr{\Lambda_tG_t(u)}=0$ implies that $G_t(u)\Lambda_t=\Lambda_t G_t(u)=0$ for $t=1,\dots,s$. A feasible point 
$u$ of \reff{nsdp} that satisfies the  conditions \reff{kkt} is called a critical point.

\subsection{Some basics in real algebraic geometry}
\label{sec:qm}

In this subsection, we review some basics in  real algebraic geometry.
For more detailed introductions, we refer  to \cite{bcr,HDLJ,kism,Las01,Lau09,niebook,sk09}.

A subset $I  \subseteq \mathbb{R}[x]$ is called an ideal of $\re[x]$
if   $I \cdot \mathbb{R}[x] \subseteq I$, $I+I \subseteq I$. 
The ideal generated by a polynomial matrix $T=(T_{ij}(x))\in \mR[x]^{r_1\times r_2}$ is defined to be the ideal generated by all its entries,  denoted as
\begin{equation*}
	\ideal{T}: \,= \, \sum\limits_{i=1}^{r_1}\sum\limits_{j=1}^{r_2}T_{ij}(x)\cdot \mR[x].
\end{equation*}
For a degree $k$, the $k$th degree truncation of $\ideal{T}$ is
\be \nonumber
\ideal{T}_k \,= \, \sum\limits_{i=1}^{r_1}\sum\limits_{j=1}^{r_2}T_{ij}(x)\cdot \mR[x]_{k-\deg(T_{ij})}.
\ee
For a   polynomial matrix tuple $H:=(H_1,\dots, H_m)$,
 the ideal generated  by $H$ is 
\begin{equation*}
	\ideal{H} := \, \ideal{H_1}+\cdots+\ideal{H_m},
\end{equation*}
and its $k$th degree truncation is
\[
\ideal{H}_k=\ideal{H_1}_k+\cdots+\ideal{H_m}_k.
\]

A polynomial $p$ is said to be a sum of squares (SOS) if
$
p=p_1^2+\dots+p_t^2$
for some $p_1,\dots,p_t \in \mathbb{R}[x]$.
The set of all SOS polynomials in $x$ is denoted by $\Sigma[x]$.
For a  degree $k$, denote the truncation
$
\Sigma[x]_{k} \, \coloneqq  \, \Sigma[x] \cap  \mathbb{R}[x]_{k} .
$
The quadratic module  generated
by  an $m\times m$ symmetric polynomial matrix $G(x)$ is the set
\be \nonumber
\qmod{G}  \coloneqq  \Big \{
\sigma+\sum_{t=1}^r v_t^TGv_t \mid
\sigma \in \Sigma[x],~v_t \in \mathbb{R}[x]^m,~r\in \N
\Big \}.
\ee
For a polynomial matrix tuple $\mathcal{G}=(G_1,\dots,G_{s})$,
the  quadratic module generated by $\mathcal{G}$ is
\be
\qmod{\mathcal{G}} \,  \coloneqq  \,  \qmod{G_1}+\cdots+ \qmod{G_s}.
\ee
Similarly, the $k$th degree truncation of $\qmod{G}$ is
\be \nonumber
\qmod{G}_{k}:=\left\{\baray{l|l}
\sigma+\sum_{t=1}^{r} v_t^TGv_t& \baray{l}  \sigma \in \Sigma[x],~v_t \in \mathbb{R}[x]^m,~r\in \N,\\
\deg(\sigma)\leq k,~\deg(v_t^TGv_t)\leq k
\earay
\earay \right\},
\ee
and the $k$th degree truncation of $\qmod{\mathcal{G}}$ is
\[
\qmod{\mathcal{G}}_{k} \,  \coloneqq  \,  \qmod{G_1}_{k}+\cdots+ \qmod{G_s}_{k},
\]
The set $\ideal{H}+\qmod{\mathcal{G}}$ is said to be Archimedean, if  there exists
$R >0$ such that $R-\|x\|^2 \in \ideal{H}+\qmod{\mathcal{G}}$.

For a truncated multi-sequence (tms) $z=(z_{\alpha})_{\alpha \in \mathbb{N}_d^n}$, it induces the Riesz functional acting on $\mathbb{R}[x]_{ d}$ as
\begin{equation} \label{Reiz:fun}
	\langle \sum_{\alpha \in \mathbb{N}_d^n} p_{\alpha} x^{\alpha},z\rangle
	\,  \coloneqq  \, \sum_{\alpha \in \mathbb{N}_d^n} p_{\alpha} z_{\alpha}.
\end{equation}
For a polynomial $q\in \mR[x]_{2k}$,
the $k$th order {\em localizing matrix} of   $q$ with respect to $z$
is the symmetric matrix $L_{q}^{(k)}[z]$ such that
\be \label{locmat:gi}
L_{q}^{(k)}[z] \, = \, \langle q\cdot [x]_{t}[x]_{t}^T, z\rangle,
\ee
where $t = k-\lceil \deg(q)/2 \rceil$,
and the functional $\langle \cdot, z\rangle$ is applied entrywise to  $q\cdot [x]_{t}[x]_{t}^T$.
When $q = 1$,  $L_{q}^{(k)}[z]$
becomes the $k$th order {\it moment matrix}
$
M_k[z]\,\coloneqq\, L_1^{(k)}[z].
$
For a polynomial matrix $T\in \mR[x]_{2k}^{r_1\times r_2}$,  the $k$th order localizing  matrix of $T$ for 
$z$ is  the block matrix
\[
L_{T}^{(k)}[z] \, \coloneqq \,
( L_{ T_{ij} }^{(k)}[z] )_{1 \le i\leq r_1, 1\le j \le r_2}.
\]
For instance, when $r_1=r_2=2$ and
\[
G(x)= \left[\begin{array}{ll}
	1-x_1x_2 &\quad x_1+x_2\\
	x_1-x_2&\quad x_1^2-x_2^2\\
\end{array}\right],
\]
we have
\[
L_{G}^{(2)}[z] =
\left[\begin{array}{lll}
	L_{1-x_1x_2}^{(2)}[z]  &\quad  L_{x_1 + x_2}^{(2)}[z]  \\
	L_{x_1 - x_2}^{(2)}[z] &\quad   L_{x_1^2-x_2^2}^{(2)}[z]
\end{array}\right],
\]
where
\[
L_{1-x_1x_2}^{(2)}[z] =
\left[\begin{array}{lll}
	z_{00} - z_{11}  &\quad  z_{10} - z_{21} &\quad z_{01} - z_{12}  \\
	z_{10} - z_{21}  &\quad  z_{20} - z_{31} &\quad z_{11} - z_{22}  \\
	z_{01} - z_{12}  &\quad  z_{11} - z_{22} &\quad z_{02} - z_{13}
\end{array}\right],
\]
\[
L_{x_1+x_2}^{(2)}[z] =
\left[\begin{array}{lll}
	z_{10} + z_{01}  &\quad  z_{20} + z_{11}  &\quad  z_{11} + z_{02}  \\
	z_{20} + z_{11}  &\quad  z_{30} + z_{21}  &\quad  z_{21} + z_{12}  \\
	z_{11} + z_{02}  &\quad  z_{21} + z_{12}  &\quad  z_{12} + z_{03}
\end{array}\right],
\]
\[
L_{x_1-x_2}^{(2)}[z] =
\left[\begin{array}{lll}
	z_{10} - z_{01}  &\quad  z_{20} - z_{11}  &\quad  z_{11} - z_{02}  \\
	z_{20} - z_{11}  &\quad  z_{30} - z_{21}  &\quad  z_{21} - z_{12}  \\
	z_{11} - z_{02}  &\quad  z_{21} - z_{12}  &\quad  z_{12} - z_{03}
\end{array}\right],
\]
\[
L_{x_1^2-x_2^2}^{(2)}[z] =
\left[\begin{array}{lll}
	z_{20} - z_{02}  &\quad  z_{30} - z_{12}  &\quad  z_{21} - z_{03}  \\
	z_{30} - z_{12}  &\quad  z_{40} - z_{22}  &\quad  z_{31} - z_{13}  \\
	z_{21} - z_{03}  &\quad  z_{31} - z_{13}  &\quad  z_{22} - z_{04}
\end{array}\right] .
\]
Note that when $r_1=r_2$ and $T$ is symmetric, the  localizing matrix $L_{T}^{(k)}[z]$ is also symmetric.

\subsection{The matrix Moment-SOS relaxations} \label{matsos}
The matrix Moment-SOS relaxations,  proposed in \cite{hdlb,schhol}, is an efficient method to solve polynomial matrix optimization globally. 
Consider the  problem \reff{nsdp}, 
where $f,h_1,\dots, h_{\ell}\in \mR[x]$,
 and each $G_t$ is an $m_t\times m_t$ symmetric  polynomial matrix.  Denote the  tuples $h=(h_1,\dots, h_{\ell})$ and $\mathcal{G}=(G_1,\dots,G_{s})$.
For an order $k>0$, the $k$th order SOS relaxation of \reff{nsdp} is
\be \label{sos:loc}
\left\{\begin{array}{cl}
	\max & \gamma \\
	\text { s.t. } & f-\gamma \in \ideal{h}_{2k}+ \qmod{\mathcal{G}}_{ 2k}.
\end{array}\right.
\ee
	The dual optimization of \reff{sos:loc} is the $k$th order moment relaxation
\be  \label{mom:loc}
\left\{ \baray{cl}
\min &  \langle f, z \rangle  \\
\st  &   L_{h_i}^{(k)}[z]=0~(i=1,\dots,\ell), \\
&  L_{G_t}^{(k)}[z] \succeq 0~(t=1,\dots,s),\\
& \langle 1, z \rangle   =  1,  \,  M_k[z] \succeq 0,\, z \in \mathbb{R}^{\mathbb{N}_{2k}^{n} } .
\earay \right.
\ee
The above relaxations can be formulated as   semidefinite programs \cite{hdlb,schhol}.  For $k=1,2, \ldots$, the sequence of relaxations \reff{sos:loc}-\reff{mom:loc}
is referred to as the matrix Moment-SOS hierarchy.

Let  $v_{\min},\,v_{k,sos}, \,v_{k,mom}$ be the optimal values of \reff{nsdp}, \reff{sos:loc} and \reff{mom:loc}, respectively. 
Note that if $f-\gamma \in \ideal{h}_{2k}+\mathrm{QM}[\mathcal{G}]_{ 2k}$, then $f(x)-\gamma\geq 0$ for all feasible $x$. This implies the monotonicity relation:
$
\cdots \leq v_{k,sos}\leq v_{k+1,sos}\leq \cdots\leq v_{\min}.
$
Similarly, if $x$ is a feasible point of \reff{nsdp}, then $[x]_{2k}$ is feasible for \reff{mom:loc}. Hence, it  holds that
$
\cdots \leq v_{k,mom}\leq v_{k+1,mom}\leq \cdots\leq v_{\min}.
$
The hierarchy is said to  have {\it finite convergence} if   $v_{k,sos}=v_{k,mom}=v_{\min}$ for all $k$ big enough.
We refer to  \cite{hdlb,hl24,hn24,schhol} for the convergence theory of this hierarchy.

In computational practice,  the flat truncation  condition is typically used to  detect  finite convergence  of the hierarchy \reff{sos:loc}--\reff{mom:loc} \cite{CF05,hdlb,HDLJ,Lau09,niebook}. Let
\[
d_K=\max\{\lceil \frac{\deg(f)}{2}\rceil,\lceil \frac{\deg(h_1)}{2}\rceil,\dots,\lceil \frac{\deg(h_{\ell})}{2}\rceil,\frac{\deg(G_{1})}{2}\rceil,\dots,\lceil \frac{\deg(G_s)}{2}\rceil\}
\]
We say the flat truncation holds at  a minimizer $z^*$ of \reff{mom:loc} if there exists 
$d_K \leq t  \leq k$ such that
\begin{equation}  	\label{flat1.3}
	\operatorname{rank} M_{t-d_{K}}[z^*]
	\, = \,  \operatorname{rank} M_{t}[z^*].
\end{equation}
When  this condition is satisfied,  we  have
$ v_{k,mom} =v_{\min}$, and minimizers of \eqref{nsdp} can be extracted.

\section{A strengthened  reformulation and  Moment-SOS relaxations} \label{sec:refor}

In this section, we study the first order optimality conditions of  \reff{int:copo}. By expressing the Lagrange multipliers in terms of the variables, we give a strengthened reformulation of \reff{int:copo}. We show that   the matrix Moment-SOS hierarchy for solving this   reformulation  has finite convergence, which is crucial for the development of our algorithm.
This technique is motivated by the recent work \cite{hnlme} and can be generalized to construct tight relaxations for general polynomial matrix optimization  of the form \reff{nsdp}, extending \cite{hnlme} where   only inequality constraints are considered.

Denote the variable vector
\[
x:=(x_{11},x_{12},\dots,x_{1n},x_{22},x_{23},\dots,x_{n-1,n},x_{nn}),
\]
and let  $X(x)$
be the $n$-by-$n$ symmetric  matrix whose $(i,j)$-th entry is $x_{ij}$ for $i\leq j$. 
 Consider the optimization problem
 \be  \label{copo}
 \left\{ \baray{rl}
 \min & f(x)  \\
 \st &\tr{X(x)}=1,\\
 & 	X(x) \succeq 0, \\
 \earay \right.
 \ee
 where $f(x)$ is a homogeneous polynomial of degree $d$.
Denote 
\[
h(x):=\tr{X(x)}-1, 
\]
and  let $f_{\min}$ be the optimal value of $\reff{copo}$. Then,  $f(x)$ is $\mathcal{S}^n_+$-copositive if and only if  $f_{\min}\geq 0$.

\subsection{A strengthened   reformulation}
We construct a strengthened reformulation for  \reff{copo} by using the optimality conditions. First, we show that the nondegeneracy condition  holds at every feasible point of  \reff{copo}.

\begin{prop}\label{ndchold}
	Let $u$ be a feasible point of  \reff{copo} with $\rank~ X(u)=r$. Then, the NDC \reff{CQ} holds at  $u$.
\end{prop}

\begin{proof}
 Let  $\{q_1,\dots,q_{n-r}\}$ be an orthonormal  basis of the kernel of  $X(u)$.	In view of  Proposition \ref{equindc}, it suffices to show that  the vectors 
 \[
 \nabla h(u),\quad \nabla X(u)^*[\frac{q_i q_j^{T}+q_j q_i^{T}}{2}]~(1\leq i\leq j \leq n-r)
 \]
  are linearly independent.  Suppose there exist  real scalars $\mu$, $\lambda_{\ij}$  such that
	\be \nonumber
	\mu \nabla h(u)+ \sum\limits_{1\leq i\leq j \leq n-r}\lambda_{ij}	\nabla X(u)^*[\frac{q_i q_j^{T}+q_j q_i^{T}}{2}]=0.
	\ee
	   Denote the matrix
	\[
	M=\sum\limits_{1\leq i\leq j \leq n-r}\lambda_{ij}\cdot \frac{q_i q_j^{T}+q_j q_i^{T}}{2}.
	\]
The equation above is equivalent to 
\[
\mu \nabla h(u)+ \nabla X(u)^*[M]=0.
\]
Then, we have that  $\mu+M_{ss}=0$ for $s=1,\dots,n$,  $~M_{st}=0$ for $s\neq t$.
	Since the vectors $q_1,\dots,q_{n-r}$ belong to the kernel of $X(u)$, it holds that 
	\[
	\baray{ll}
	X(u)\cdot  M&=X(u)\cdot (\sum\limits_{1\leq i\leq j \leq n-r}\lambda_{ij}\cdot \frac{q_i q_j^{T}+q_j q_i^{T}}{2})\\
	&=\sum\limits_{1\leq i\leq j \leq n-r}\lambda_{ij}\cdot \frac{X(u)q_i q_j^{T}+X(u)q_j q_i^{T}}{2}=0,
	\earay
	\]
which implies
\[
0=(X(u)\cdot  M)_{ss}=u_{ss}M_{ss}=-\mu u_{ss}~(s=1,\dots,n).
\]
Since $u_{11}+\cdots+u_{nn}=1$, there exists at least one  $1\leq k_0\leq n$ such that $u_{k_0k_0}\neq 0$. Thus, we have $\mu=0$ and  $M=0$. For any $1 \leq i_0 < j_0 \leq n-r$, we get 
\[
\baray{ll}
0=q_{i_0}^TMq_{j_0}&=q_{i_0}^T(\sum\limits_{1 \leq i \leq j \leq n-r} \lambda_{ij}\cdot \frac{q_i q_j^{T}+q_j q_i^{T}}{2})q_{j_0}\\
&=\sum\limits_{1 \leq i \leq j \leq n-r} \lambda_{ij}\cdot \frac{q_{i_0}^Tq_i q_j^{T}q_{j_0}+q_{i_0}^Tq_j q_i^{T}q_{j_0}}{2}=\frac{1}{2}\lambda_{i_0j_0}.\\
\earay 
\]
Similarly, we have $0=q_{i_0}^TMq_{i_0}=\lambda_{i_0i_0}$ for $1\leq i_0\leq n-r$.
These imply that $\lambda_{ij}=0$ for all $1\leq i\leq j \leq n-r$, which completes the proof.
\end{proof}

Let $u$ be a local minimizer of \reff{copo}. By Proposition \ref{ndchold} and Theorem \ref{KKT:cond},  there exist $\mu\in \mR$ and $\Lambda \in \mathcal{S}_{+}^n$  such that $\tr{\Lambda X(u)}=0$ and
\be \label{kktnco}
 \nabla	f(u)-\mu \nabla h(u)-\nabla X(u)^*[\Lambda] = 0.
\ee
By Euler's identity for homogeneous polynomials, we have  $x^T\nabla	f(u)=d\cdot f(x)$.
Note that 
\[
x^T \nabla h(u)=\tr{X(u)}=1,~x^T(\nabla X(u)^*[\Lambda])=\tr{\Lambda X(u)}=0.
\]
Multiplying   both sides of   \reff{kktnco} by $x^T$, we  obtain 
\[
0=x^T(\nabla	f(u)-\mu \nabla h(u)-\nabla X(u)^* [\Lambda])=d\cdot f(u)-\mu,
\]
which implies that $\mu=d\cdot f(u)$.
The equation  \reff{kktnco} is equivalent to 
\[
\partial_{x_{ii}}f(u)-\mu-\Lambda_{ii}=0~(i=1,\dots,n),
\]
\[
\partial_{x_{ij}}f(u)-2\Lambda_{ij}=0~(1\leq i<j\leq n).
\]
Then, we have

\[
\Lambda_{ii}=\partial_{x_{ii}}f(u)-d\cdot f(u)~(i=1,\dots,n),
\]
\[
\Lambda_{ij}=\frac{1}{2}\partial_{x_{ij}}f(u)~(1\leq i<j\leq n).
\]
Let $\Theta(x)$ be the $n$-by-$n$ symmetric polynomial matrix with entries given as above, i.e., 
{\small
\be \label{exp:theta1}
\Theta(x)= \left[\begin{matrix}
	\partial_{x_{11}}f(x)-d\cdot f(x) &\frac{1}{2}\partial_{x_{12}}f(x)&\cdots &\frac{1}{2}\partial_{x_{1n}}f(x)\\
	\frac{1}{2}\partial_{x_{12}}f(x)&\partial_{x_{22}}f(x)-d\cdot f(x)&\cdots &\frac{1}{2}\partial_{x_{2n}}f(x)\\
\vdots 	&\vdots &\ddots &\vdots \\
\frac{1}{2}\partial_{x_{1n}}f(x)&\frac{1}{2}\partial_{x_{2n}}f(x)&\cdots& \partial_{x_{nn}}f(x)-d\cdot f(x)
\end{matrix}\right].
\ee
}

Since $X(u)\succeq 0$, $\Theta(u) \succeq 0$, and $\tr{\Theta(u)X(u)}=0$, we have  $X(u)\Theta(u)=0$. 
In the following, we give an estimate on the Frobenius norm of  $X(u)$. Let $\lambda_1,\dots,\lambda_n$ be the eigenvalues of $X(u)$. Since $X(u)\succeq 0$ and $\tr{X(u)}=1$, it holds that
\[
\lambda_1\geq 0,\dots,\lambda_n \geq 0,~ \lambda_1+\cdots+\lambda_n=1.
\]
Then, we have that $0\leq \lambda_i \leq 1$ and 
\[
\|X(u)\|^2_F=\lambda_1^2+\cdots+\lambda_n^2\leq \lambda_1+\cdots+\lambda_n=1,
\]
\[
\|X(u)\|^2_F=\frac{1}{n}\cdot (1^2+\cdots+1^2)(\lambda_1^2+\cdots+\lambda_n^2)\geq \frac{1}{n}\cdot (\lambda_1+\cdots+\lambda_n)^2=\frac{1}{n},
\]
where the last inequality follows from Cauchy's inequality.

Since the feasible set of \reff{copo} is compact, its optimal value  is achievable. Hence,  \reff{copo} is equivalent to the following strengthened reformulation:
\be  \label{copo:equ}
\left\{ \baray{rl}
\min & f(x)  \\
\st &  \tr{X(x)}=1,~X(x)\Theta(x)=0, \\
   &    X(x)\succeq 0,~\Theta(x)\succeq 0,\\
   &  1\geq \|X(x)\|_F^2,~\|X(x)\|_F^2\geq  \frac{1}{n}.\\
\earay \right.
\ee
Clearly, the optimal value of \reff{copo:equ} is equal to $f_{\min}$.

\subsection{The matrix Moment--SOS hierarchy of \reff{copo:equ}}
Let $d_{0}:=\lceil \frac{d+1}{2} \rceil$. 
We apply the matrix Moment-SOS relaxations  to solve \reff{copo:equ} (see Section \ref{matsos}). For an order $k\geq  d_0$, the $k$th order SOS relaxation of \reff{copo:equ} is 

\be  \label{sos}
\left\{ \baray{rll}
\max & \gamma &\\
\st &f-\gamma \in &\ideal{\tr{X}-1,X\Theta}_{2k}\\
& &+\qmod{X,\Theta,1-\|X\|_F^2,\|X\|_F^2-\frac{1}{n}}_{2k}.
\earay \right.
\ee
The dual problem of \reff{sos} is the $k$th order moment relaxation:
\be  \label{mom}
\left\{ \baray{rl}
\min & \langle f, z\rangle  \\
\st & L^{(k)}_{\tr{X}-1}(z)=0,~L^{(k)}_{X \Theta}(z)=0,\\
&  L_{X}^{(k)}(z) \succeq 0,~L_{\Theta}^{(k)}(z) \succeq 0,\\
&L_{1-\|X\|_F^2}^{(k)}(z) \succeq 0,~L_{\|X\|_F^2-\frac{1}{n}}^{(k)}(z) \succeq 0,\\
&\langle 1, z\rangle=1,~M_k(z) \succeq 0,~ z \in \mathbb{R}^{\mathbb{N}_{2k}^{\sigma(n)} } .
\earay \right.
\ee
Let  $f_{k,sos}$ and $f_{k,mom}$ denote the optimal values of  \reff{sos} and \reff{mom}, respectively. 

Denote the ideal
\be \label{equI}
I:=\ideal{h,~X\Theta}.
\ee
The set of complex critical points of \reff{copo} is given by
\be \label{kktcri}
V_{\mC}(I):=\left\{
	x \in \mathbb{C}^{\sigma(n)} :
	h(x)=0,~X(x)\Theta(x)=0 \\
	\right\}.
\ee
We show that $f$  attains a constant real value on each irreducible subvariety of $V_{\mC}(I)$ that contains at least one real point.

\begin{lemma}\label{kktfini}                    Let $W$ be an  irreducible subvariety of $V_{\mC}(I)$ with $W\cap \mR^{\sigma(n)} \neq\emptyset$. Then, $f$  attains a constant real value on $W$.
\end{lemma}

\begin{proof}
   Denote the Lagrangian function
   $$
   \mathcal{L}(x)=f(x)-d\cdot f(x)h(x)-\tr{\Theta(x)X(x)}. 
   $$
   Note that for $x\in V_{\mC}(I)$, we have $\mathcal{L}(x)=f(x)$.
    Let  $x^{(1)}$ and $x^{(2)}$ be two arbitrary distinct points in $W$. We show that $f(x^{(1)})=f(x^{(2)})$. 
   
   Since  $W$ is  path-connected  in the strong topology on $\mathbb{C}^{\sigma(n)}$ (\cite[4.1.3]{sawc}), there exists a piecewise-smooth curve  $x(t)$ $(0 \leq t \leq 1)$ such that
   \[
   \{x(t):~~0 \leq t \leq 1\}\subseteq W,~x(0)=x^{(1)},~ x(1)=x^{(2)},
   \] 
   and  a partition $0 =t_0 < \cdots < t_s= 1$ such that $x(t)$ is smooth on each interval $(t_i, t_{i+1})$ for $i=0,\dots,s-1$.
   For $r=n,n-1,\dots,0$,  define the set $T_r$ recursively as 
   \[
   T_r=\{t\in (t_i, t_{i+1})\mid \rank\, X(x(t)) =r\}\backslash \cup_{j=r+1}^n cl(T_j).
   \] 	
      Let $\bar{t}$ be an arbitrary point in $T_r$.
      
       For $r=n$,  the matrix $X(x(\bar{t}))$ is invertible, and there exists a sufficiently small $\epsilon>0$ such that $G(x(t))$ remains invertible for all $t\in (\bar{t}-\epsilon,\bar{t}+\epsilon)$.  The equation $X(x(t))\Theta(x(t))=0$ implies that $\Theta(x(t))$ is identically zero for all $t\in (\bar{t}-\epsilon,\bar{t}+\epsilon)$. Then, we have 
   \[
   \baray{ll}
   \frac{\mathrm{d} f}{\mathrm{d}t}(x(\bar{t}))&=\frac{\mathrm{d} \mathcal{L}}{\mathrm{d}t}(x(\bar{t}))\\
   &=(\nabla f(x(\bar{t}))-d\cdot f(x(\bar{t}))\nabla h(x(\bar{t}))-\nabla X(x(\bar{t}))^*[\Theta(x(\bar{t}))])^T\nabla x(\bar{t})\\
   &\quad -d\cdot h(x(\bar{t}))\nabla f(x(\bar{t}))^T\nabla x(\bar{t})-\tr{ X(x(\bar{t}))\frac{d\Theta(x(\bar{t}))}{dt}} \\
   &=0,\\
   \earay
   \] 
where the last equality follows from the expression of $\Theta(x)$  in \reff{exp:theta1}.
Hence, the gradient of $f(x(t))$ at $\bar{t}$ is $0$.
   Since $\bar{t}$ is arbitrary in  $S_n$, it follows that $f(x(t))$ has  zero gradient for all $t\in cl(T_n)\backslash \{t_i,t_{i+1}\}$.

   For $r<n$,  we have $ X(x(\bar{t}))=r$. Up to a suitable permutation of the rows and columns of $G(x)$, we can assume that
   \[
   X(x)=\left[\begin{array}{cc}
   	G_1(x)& G_2(x) \\
   	G_2(x)^T &G_3(x)\\
   \end{array}\right],~ 
   \]
   where $G_1(x(\bar{t}))$ is an $r$-by-$r$ invertible complex symmetric matrix. Let
   \[
   Q(x)=\left[\begin{array}{cc}
   	I_{r}& -G_1(x)^{-1}G_2(x) \\
   	0 &I_{n-{r}}\\
   \end{array}\right],
   \]
 \[
   S(x)=G_3(x)-G_2(x)^TG_1(x)^{-1}G_2(x).
   \]
 One can see that
   \[
   Q(x)^TX(x)Q(x)=\left[\begin{array}{cc}
   	G_1(x)& 0 \\
   	0 &S(x)\\
   \end{array}\right].
   \]
   Since $\bar{t}\notin  \cup_{j=r+1}^n cl(T_j)$, there exists $\epsilon>0$ such that  $\rank X(x(\bar{t}))$ is maximal in the interval $ (\bar{t}-\epsilon,\bar{t}+\epsilon)$; that is,
     \[
     \rank\, X(x(t))=\rank\, G_1(x(t))=r,\,\,\ \forall \, t\in (\bar{t}-\epsilon,\bar{t}+\epsilon).
     \] 
    Hence, we have $S(x(t))=0$ for all $t\in (\bar{t}-\epsilon,\bar{t}+\epsilon)$. 
  We write that
   \[
   Q(x(t))^{-1}\Theta(x(t)) Q(x(t))^{-T}=\left[\begin{array}{cc}
   	\Lambda_1(t) & \Lambda_2(t) \\
   	\Lambda_2(t)^T & \Lambda_3(t)\\
   \end{array}\right],
   \]
   where $\Lambda(t)$ is an $r$-by-$r$ symmetric polynomial matrix in $t$.
   Then, it holds that
   \[
   \baray{ll}	
   0=X(x(t))\Theta(x(t))&=Q(x(t))^TX(x(t))Q(x(t))Q(x(t))^{-1}\Theta(x(t)) Q(x(t))^{-T} \\
   &=\left[\begin{array}{cc}
   	G_1(x(t))\Lambda_1(t) & G_1(x(t))\Lambda_2(t) \\
   	S(x(t))\Lambda_2(t)^T & S(x(t))\Lambda_3(t)\\
   \end{array}\right],\\
   \earay
   \]
which implies
  \[
  \Lambda_1(t)=\Lambda_2(t)=0,\,\, \forall t\in (\bar{t}-\epsilon,\bar{t}+\epsilon).
  \]
    For convenience, we  write   $G_i(x(t))$ as $G_i$. Then, we get
   \[
   \baray{ll}
   \Theta(x(t)) &=Q(x(t))\left[\begin{array}{cc}
   	0 & 0 \\
   	0 & \Lambda_3(t)\\
   \end{array}\right]Q(x(t))^{T}\\
   &=\left[\begin{array}{cc}
   	G_1^{-1}G_2\Lambda_3(t)G_2^TG_1^{-T} & -G_1^{-1}G_2\Lambda_3(t) \\
   	-(G_1^{-1}G_2\Lambda_3(t))^T & \Lambda_3(t)\\
   \end{array}\right].\\
   \earay
   \]	
Following the ideal as in \cite[Lemma 4.2 ]{hnlme}, we have 
   $$
   	\tr{X(x(t))\frac{d\Theta(x(t))}{dt}} =0.
   $$
   This implies  that
   \[
   \baray{ll}
   \frac{\mathrm{d} f}{\mathrm{d}t}(x(\bar{t}))&=\frac{d\mathcal{L}}{dt}(x(\bar{t}))\\
   &=(\nabla f(x(\bar{t}))-d\cdot f(x(\bar{t}))\nabla h(x(\bar{t}))-\nabla X(x(\bar{t}))^* [\Theta(x(\bar{t}))])^T\nabla x(\bar{t})\\
   &\quad \quad -d\cdot h(x(\bar{t}))\nabla f(x(\bar{t}))^T\nabla x(\bar{t})-\tr{X(x(\bar{t}))\frac{d\Theta(x(\bar{t}))}{dt}} \\
   &=0\\
   \earay	
   \]
   Since $\bar{t}$ is arbitrary in  $T_{r}$,  we know that  $f(x(t))$ has  zero gradient
   for all $t\in cl(T_{r})\backslash \{t_i,t_{i+1}\}$.

   Note that
   $
   (t_i,t_{i+1})=\cup_{r=0}^n cl(T_{r})\backslash \{t_i, t_{i+1}\}
   $. 
 We know that   $f(x(t))$ has zero gradient on  $(t_i,t_{i+1})$. Hence,   by integration, it holds that 
  \[
  f(x(t_{i+1})) - f(x(t_i)) =
  \int_{t_i}^{ t_{i+1} } \frac{\mathrm{d} f(x(t))}{\mathrm{d}t}  dt  = 0.
  \]
   Since it holds for all intervals $(t_i,t_{i+1})$, we have
  \[
  f(x^{(1)})=f(x(t_0))=f(x(t_1))=\cdots=f(x(t_s))=f(x^{(2)}),
  \]
 which implies that $f$  attains a constant  value on $W$. Since $W$ contains at least one real point, this value must be real.

\end{proof}

In the following,  we prove that  the strengthened hierarchy  \reff{sos}--\reff{mom}   has finite convergence.

\begin{thm}\label{thm:fin}
 Suppose  $f(x)$ is a  homogeneous polynomial of degree $d$. Then, we have:
 
 \bit
 
 \item[(i)] The moment relaxation \reff{mom} is feasible and  its optimal value is attainable.

 \item[(ii)] There exists an order $k_0>0$ such that $f_{k,sos}= f_{k,mom}=f_{\min}$ for all $k\geq k_{0}$.
 
 \eit 
\end{thm}

\begin{proof}
	(i) Note that  the feasible set of \reff{copo} is compact, so it  has at least  one minimizer $x^*$. Since  the NDC \reff{CQ} holds at $x^*$  (see Proposition \ref{ndchold}),  $x^*$ is   feasible  for \reff{copo:equ} and $[x^*]_{2k}$  is   feasible  for  \reff{mom}. 
	Let  $z$ be a feasible point  of  \reff{mom}. Note that $L_{1-\|X\|_F^2}^{(k)}(z)\succeq 0$, $M_k(z) \succeq 0$, and  
	\[
	1-\|x\|^2 = 1-\|X(x)\|_F^2+\sum\limits_{1\leq i<j\leq n} x_{ij}^2.
	\]
	We have $L_{1-\|x\|^2}^{(k)}(z)\succeq 0$, and 
	\[
    \langle \|x\|^{2k}, z\rangle \leq  \langle \|x\|^{2(k-1)}, z\rangle \leq \cdots \leq  \langle 1, z\rangle=1.
	\]
	It implies that
	\[
	\|M_k(z)\|_{F}^2\leq (\tr{M_k(z)})^2 \leq \sum\limits_{i=0}^{k} \langle \|x\|^{2i}, z\rangle \leq 2(k+1).
	\]
Hence, the feasible set of \reff{mom}  is  compact and  the optimal value of \reff{mom}  is attainable.

	(ii) Without loss of generality, we  assume   $f_{\min}=0$, up to shifting $f$ by a constant. 
	Let $V_{\mC}(I):=W_1 \cup \cdots \cup W_s$ be an irreducible decomposition of  $V_{\mC}(I)$. By  Lemma \ref{kktfini},  $f(x)$ attains a constant real value on  $W_i$ if  $W_i \cap \mR^{\sigma(n)} \neq\emptyset$. Then, we know that  $f(x)$ attains finitely many distinct real
	values on $V_{\mC}(I)$,  ordered as $v_1<v_2<\cdots<v_{\ell}$. 
	Let $\mathcal{K}_0$ be the union of all  subvarieties $W_i$ such that $W_i\cap \mathbb{R}^{\sigma(n)}=\emptyset$, and let $\mathcal{K}_i$ be the union of all remaining $W_i$ on which $f$ attains the constant value $v_i$. 
	Then,  $f$ is identically equal to  $v_i$ on $\mathcal{K}_i$ for $i=1, \ldots,\ell$, and the  complex varieties $\mathcal{K}_0, \mathcal{K}_1, \ldots, \mathcal{K}_{\ell}$ satisfy
	\[
	V_{\mC}(I)=\mathcal{K}_0 \cup \mathcal{K}_1 \cup \cdots \cup \mathcal{K}_{\ell},~~~~
	\mathcal{K}_0 \cap \mathbb{R}^n=\emptyset.
	\]   
	By the primary decomposition of the ideal $I$ \cite{bcr},    there exist ideals $I_0,I_1,\dots,I_{\ell} \subseteq \mR[x]$ such that 
	\[
	I=I_0 \cap I_1 \cap \ldots \cap I_{\ell},\quad
	\mathcal{K}_i=V_{\mathbb{C}}(I_i)~~(i=0, \ldots, \ell).
	\]
	Assume that $v_{\ell_0}=f_{\min}$ for some  $1\leq \ell_0\leq \ell$.
	
	 Note that $V_{\mathbb{C}}\left(I_0\right) \cap \mathbb{R}^n=\emptyset$. It follows from \cite[Corollary 4.1.8]{bcr} that there exists $\tau_0 \in \Sigma[x]$ such that $1+\tau_0 \in I_0$. Let 
	\[
	\sigma_0:=\frac{1}{4}(f+1)^2+\frac{\tau_0}{4}(f-1)^2.
	\]
	Then, we have that $\sigma_0\in \Sigma[x]$ and 
	\[
	f-\sigma_0=\frac{1}{4}(f+1)^2-\frac{1}{4}(f-1)^2-\sigma_0=-\frac{1+\tau_0}{4}(f-1)^2\in I_0.
	\]
	
	For $i=1,\dots,\ell_0-1$, we have $v_i<f_{\min}=0$.  This implies that the real points in $V_{\mathbb{C}}(I_i)$ are not feasible for \reff{copo}, i.e.,
	\[
	\{x\in \mathbb{R}^{\sigma(n)}: p(x)=0,\,\,\forall p\in I_i\}\cap \{x\in \mR^{\sigma(n)}:  X(x)\succeq 0\}=\emptyset.
	\]
	By \cite[Proposition 9]{sk09},  there exist polynomials $d_1,\dots,d_t\in \qmod{X}$ such that 
	\[
	\{x\in \mR^{\sigma(n)}:X(x) \succeq 0\}=\{x\in \mR^{\sigma(n)}:d_1(x)\geq 0,\dots,d_t(x)\geq 0\}.
	\]
Then, it follows from  \cite[Corollary 4.4.3]{blo} that there exists 
$
\phi=\sum_{\alpha\in\{0,1\}^t}\phi_{\alpha}d_1^{\alpha_1}\cdots d_t^{\alpha_t}$ with $ \phi_{\alpha}\in \Sigma[x]
$ 
 such that $2+\phi\in I_i$. 
	Note that $1+\phi(x)>0$ on the feasible set of \reff{copo} and 
 the quadratic module  $\qmod{X,\Theta,1-\|X\|_F^2,\|X\|_F^2-\frac{1}{n}}$   is Archimedean. By Putinar's Positivstellensatz \cite{putinar1993positive}, we have
\[
1+\phi\in \qmod{X,\Theta,1-\|X\|_F^2,\|X\|_F^2-\frac{1}{n}}.
\] 
 Let
	\[
	\sigma_i=\frac{1}{4}(f+1)^2+\frac{1+\phi}{4}(f-1)^2.
	\]
	Then, we know that 
	\[
	\sigma_i \in \qmod{X,\Theta,1-\|X\|_F^2,\|X\|_F^2-\frac{1}{n}},
	\]
	\[
	f-\sigma_i=\frac{1}{4}(f+1)^2-\frac{1}{4}(f-1)^2-\sigma_i=-\frac{2+\phi}{4}(f-1)^2\in I_i.
	\]
	
	For $i=\ell_0$,  we know  that $f$ is identically  equal to $0$ on $V_{\mathbb{C}}(I_{\ell_0}).$
	By Hilbert's Strong Nullstellensatz \cite{blo}, there exists an integer $\eta>0$ such that $f^{\eta} \in I_{\ell_0}$. For $\epsilon>0$, let
	\[
	s_{\ell_0}^{\epsilon}=\sqrt{\epsilon} \sum_{j=0}^{\eta-1}\binom{\frac{1}{2}}{j} \epsilon^{-j} f^j,~~~ \sigma_{\ell_0}^{\epsilon}=s_{\ell_0}^2.
	\]
Then, for any $\epsilon >0$, we have that
	\be \label{epl0}
	f+\epsilon-\sigma_{\ell_0}^{\epsilon}=\sum_{j=0}^{\eta-2} b_j^{\epsilon} f^{\eta+j}\in I_{\ell_0},
	\ee
	where  $b_j^{\epsilon}$ are real scalars depending  on $\epsilon$.

	For $i=\ell_0+1, \ldots, \ell$, we know that $v_i>f_{\min}=0$ and
	$v_i^{-1} f-1$ is identically  zero on $V_{\mathbb{C}}(I_i).$
It follows from Hilbert's Strong Nullstellensatz \cite{blo} that there exists $\eta_i \in \mathbb{N}$ such that $ \left(v_i^{-1} f-1\right)^{\eta_i} \in I_i$. Let
	$$
	s_i=\sqrt{v_i} \sum_{j=0}^{\eta_i-1}\binom{\frac{1}{2}}{j} \left(v_i^{-1} f-1\right)^j,~ \sigma_i=s_i^2.
	$$
Similarly to  \reff{epl0}, we have 
	$	f-\sigma_i\in I_i$.

Note that  the complex varieties $\mathcal{K}_0, \mathcal{K}_1, \ldots, \mathcal{K}_\ell$ are disjoint.  By \cite[Lemma 3.3]{niejac}, there exist  $a_0, \ldots, a_{\ell} \in \mathbb{R}[x]$ such that
	\[
	a_0^2+\cdots+a_{\ell}^2-1 \in I,\quad 
	a_i \in \bigcap_{i\neq j \in \{0,\dots,\ell\}} I_{j}~~(i=0,\dots,\ell).
	\]
		For $\epsilon>0$, denote 
	\[
	\sigma_{\epsilon}=\sigma_{\ell_0}^{\epsilon}a_{\ell_0}^2+\sum_{\ell_0 \neq i \in\{0, \ldots, \ell\}}\left(\sigma_i+\epsilon\right) a_i^2.
	\]
	Then, we have 
	$$
	\begin{aligned}
		f+\epsilon-\sigma_\epsilon= & (f+\epsilon)\left(1-a_0^2-\cdots-a_{\ell}^2\right) \\
		& +\left(f+\epsilon-\sigma_{\ell_0}^{\epsilon}\right) a_{\ell_0}^2+\sum_{\ell_0 \neq i \in\{0, \ldots, \ell\}}\left(f-\sigma_i\right) a_i^2 .
	\end{aligned}
	$$
Since  $f-\sigma_i \in I_i$ for each $i \neq \ell_0$,  there exists $k_1>0$ such that 
\[
	(f+\epsilon)\left(1-a_{0}^2-\cdots-a_r^2\right) \in I_{2k_1},\quad (f-\sigma_i) a_i^2 \in I_{2k_1}.
\] 
 Multiplying    both sides of \reff{epl0} by $a_{\ell_0}^2$, we obtain 
	\[
	(f+\epsilon-\sigma_{\ell_0}^{\epsilon})a_{\ell_0}^2=\sum_{j=0}^{\eta-2} b_j^{\epsilon} f^{\eta+j}a_{\ell_0}^2.
	\]
Hence, there exists $k_2>0$ such that  $(f+\epsilon-\sigma_{\ell_0}^{\epsilon})a_{\ell_0}^2\in I_{2k_2}$ for all $\epsilon>0$. From the expression of $\sigma_{\epsilon}$, there exists $k_3>0$ such that \[
\sigma_{\epsilon} \in I_{2k_3} + \qmod{X,\Theta,1-\|X\|_F^2,\|X\|_F^2-\frac{1}{n}}_{2k_3}.
\]
Then, for $k\geq \max\{k_1,k_2,k_3\}$, 	we have
	\[
	f-f_{\min}+\epsilon \in I_{2k}+ \qmod{X,\Theta,1-\|X\|_F^2,\|X\|_F^2-\frac{1}{n}}_{2k}.
	\]
This implies that when $k$ is sufficiently large, we have $f_{k,sos} \geq f_{\min}-\epsilon$ for all $\epsilon>0$. Since $f_{k,sos} \leq f_{\min}$ for all $k$, it follows that  $f_{k,sos}=f_{k,mom}=f_{\min}$ for all $k$ sufficiently large.

\end{proof}

\section{An algorithm for testing $\mathcal{S}^n_{+}$-copositivity}
\label{sec:alo}

In this section, we present our algorithm for testing $\mathcal{S}^n_{+}$-copositivity. We show that the algorithm terminates in  finitely many iterations for any homogeneous polynomial $f$, either  providing a  certificate that $f$ is $\mathcal{S}_+^n$-copositive, or returning  a refutation $0\neq u\in \mR^{\sigma(n)}$ satisfying $X(u)\succeq 0$ such that $f(u)<0$.

\vspace{.2cm}

The algorithm is given below.

\begin{alg}
	\label{ag:CopoSDP} \rm
	Testing $\mathcal{S}_+^n$-copositivity of  homogeneous polynomials. 
	
	\begin{description}
			\item [Input]  A homogeneous polynomial $f(x)$ of degree $d$.
		\item [Step~0]
	 Choose a generic vector $\xi \in \mathbb{R}^{\mathbb{N}_d^{\sigma(n)}}$. Let $k:=\lceil \frac{d+1}{2}\rceil$.
		
		\item [Step~1]
	 Solve the semidefinite relaxation pair \reff{sos}--\reff{mom}. If the  optimal value $f_{k,mom} \geq 0$, output that $f$ is $\mathcal{S}_+^n$-copositive  and stop; otherwise, go to Step 2.
		
		\item[Step~2]
	Solve the semidefinite program
	
\be  \label{momran}
\left\{ \baray{rl}
\min & \langle \xi^T[x]_d, w\rangle  \\
\st & L^{(k)}_{\tr{X}-1}(w)=0,\\
&  L_{X}^{(k)}(w) \succeq 0,~L_{f_{k,mom}-f}^{(k)}(w) \succeq 0,\\
&L_{1-\|X\|_F^2}^{(k)}(w) \succeq 0,~L_{\|X\|_F^2-\frac{1}{n}}^{(k)}(w) \succeq 0,\\
&\langle 1, w\rangle=1,~M_k(w) \succeq 0,~ w \in \mathbb{R}^{\mathbb{N}_{2k}^{\sigma(n)} } .
\earay \right.
\ee
If \reff{momran} is feasible, compute an optimizer $w^*$ and go to Step 3;  otherwise, let $k:=k+1$ and go to Step 1.

\item[Step~3] Let $u=(w^*_{e_1}, \dots,w^*_{e_{\sigma(n)}})$. If $f(u)<0$, output that $f(x)$ is not $\mathcal{S}_+^n$-copositive, return $u$ and  stop;  otherwise, let $k:=k+1$ and go to Step 1.

\item [Output] A certificate that $f$ is $\mathcal{S}_+^n$-copositive   or a refutation $0\neq u\in \mR^{\sigma(n)}$  such that $X(u)\succeq 0$ and $f(u)<0$.

	\end{description}
\end{alg}

\vspace{.3cm}

\begin{remark}\label{rem1}

We make the following remarks about  Algorithm \ref{ag:CopoSDP}:
\bit

\item[(i)]	In Step 0,   the vector $\xi \in \mathbb{R}^{\mathbb{N}_d^{\sigma(n)}}$ is  said to be generic  if it lies in the input space excluding a subset of measure zero. In numerical experiments, we can choose $\xi$ as  a random vector whose entries are independently sampled from the standard normal distribution.

\item[(ii)]   Note that  $f_{k,mom}$ is the optimal value of the semidefinite relaxation \reff{mom}.  When solving it numerically, rounding errors may occur. Therefore, we treat $f_{k,mom} \geq 0$ if $f_{k,mom}\geq -10^{-5}$.

\item[(iii)]   We cannot test copositivity simply by solving the hierarchy \reff{sos}--\reff{mom}, because  if $f$ is not $\mathcal{S}_+^n$-copositive, we may not be able to certify that $f_{k,mom}=f_{\min}<0$ at some relaxation order $k$, even if finite convergence occurs. This is because the flat truncation condition \reff{flat1.3} may not be satisfied. In numerical practice, if the condition \reff{flat1.3} holds for the minimizer of  \reff{mom}, we can also detect $\mathcal{S}_+^n$-copositivity and  terminate the iteration, as the minimizers of \reff{copo} can  be extracted (see Section~\ref{matsos}).


\eit
 
\end{remark}


In the following, we show that Algorithm \ref{ag:CopoSDP}  always terminates in finitely many
iterations, i.e., testing $\mathcal{S}_+^n$-copositivity  can be done by solving a finite number of semidefinite programs.

\begin{thm}\label{finitethm}
Suppose $f(x)$ is a homogeneous polynomial of degree $d$. Then,	Algorithm \ref{ag:CopoSDP} terminates in finitely many iterations. To be more specific, we have:
	
	\bit 
	
	\item[(i)] If $f(x)$ is $\mathcal{S}^n_+$-copositive,  we have $f_{k,mom}\geq 0$ for all sufficiently large $k$. Consequently, Algorithm \ref{ag:CopoSDP} outputs that $f$ is $\mathcal{S}^n_+$-copositive.

	\item[(ii)] If $f(x)$ is not $\mathcal{S}^n_+$-copositive, 	Algorithm \ref{ag:CopoSDP}  returns a
	nonzero vector $ u\in \mR^{\sigma(n)}$ satisfying $X(u)\succeq 0$ such that $f(u)<0$.

	\eit 
\end{thm}

\begin{proof}
	By Theorem \ref{thm:fin} (ii),  there  exists an integer $k_0>0$ such that $f_{k,sos}= f_{k,mom}=f_{\min}$ for all $k\geq k_{0}$. 
	
	(i) Note that the vector $u=(w^*_{e_1}, \dots,w^*_{e_{\sigma(n)}})$ is feasible for \reff{copo}. Since $f(x)$ is $\mathcal{S}^n_+$-copositive, we  know that $f(u)\leq f_{k,mom}<0$ if $f_{k,mom}<0$. Hence, Algorithm \ref{ag:CopoSDP} terminates at Step 1 when $k$ is sufficiently large, since $f_{k_0,mom}=f_{\min}\geq 0$.

	(ii) Since $f(x)$ is not $\mathcal{S}^n_+$-copositive,  we have $ f_{k,mom}=f_{\min}<0$ for  all $k\geq k_{0}$,  and   \reff{momran} is equivalent to 
	
	\be  \label{momran:1}
	\left\{ \baray{rl}
	\min & \langle  \xi^T[x]_d, w\rangle  \\
	\st & L^{(k)}_{\tr{X}-1}(w)=0,\\
	&  L_{X}^{(k)}(w) \succeq 0,~L_{f_{\min}-f}^{(k)}(w) \succeq 0,\\
	&L_{1-\|X\|_F^2}^{(k)}(w) \succeq 0,~L_{\|X\|_F^2-\frac{1}{n}}^{(k)}(w) \succeq 0,\\
	&\langle 1, w\rangle=1,~M_k(w) \succeq 0,~w \in \mathbb{R}^{\mathbb{N}_{2k}^{\sigma(n)}}.
	\earay \right.
	\ee
The above is the $k$th order moment relaxation for  the following problem:
	\be  \label{copo:equ:1}
	\left\{ \baray{rl}
	\min & \xi^T[x]_d  \\
	\st &  \tr{X(x)}=1,\\
	&    X(x)\succeq 0,~f_{\min}-f(x) \geq  0,\\
	&  1\geq \|X(x)\|_F^2,~\|X(x)\|_F^2\geq  \frac{1}{n}.\\
	\earay \right.
	\ee
The moment reformulation of \reff{copo:equ:1} is 
\be  \label{copo:equ:3}
\left\{ \baray{rl}
\min & \langle \xi^T[x]_d, v\rangle  \\
\st &  \langle 1, v\rangle=1,\\
&   v\in \mathscr{R}(K^{\prime}),\\
\earay \right.
\ee
where $K^{\prime}$ is the feasible set of \reff{copo:equ:1}, and $\mathscr{R}(K^{\prime})$ is the moment cone, i.e., the set of all truncated multi-sequences $v \in \mR^{\mathbb{N}^{\sigma(n)}_{d}}$ that admit a positive Borel measure supported on  $K^{\prime}$.
Similarly as in Theorem \ref{thm:fin} (i), we can show that every feasible point $v$ of \reff{copo:equ:3} satisfies $\|v\|^2\leq 2(d+1)$. Hence, \reff{copo:equ:3} is equivalent to 
\be  \label{copo:equ:4}
\left\{ \baray{rl}
\min & \langle \xi^T[x]_d, v\rangle  \\
\st &  \langle 1, v\rangle=1,\\
&  \|v\|^2\leq 2(d+1),\,\, v\in \mathscr{R}(K^{\prime}).\\
\earay \right.
\ee
The feasible set of  \reff{copo:equ:4} is a nonempty compact convex set. Hence,  \reff{copo:equ:4}  has a unique minimizer if and only if $\xi$ is a singular normal vector of the feasible set (see \cite[Section 2.2]{consch}). Let $\Omega$ be the set of all singular normal vectors.  Then, the set $\Omega$ has zero Lebesgue measure in the input space (cf. \cite[Section 2.2.4]{consch}).  Thus,  \reff{copo:equ:4}  has a unique minimizer for all $\xi\in  \mathbb{R}^{\mathbb{N}_d^{\sigma(n)}} \backslash \Omega$. This implies that \reff{copo:equ:1} also has a unique minimizer, which we denote  by  $x^*$.

 Let $w^{(k)}$  be the minimizer of \reff{momran:1} at the  order $k$. 
By \cite[Proposition 9]{sk09},  there exist  $d_1,\dots,d_t\in \qmod{X}$ such that 
\[
\{x\in \mR^{\sigma(n)} :X(x) \succeq 0\}=\{x\in \mR^{\sigma(n)} :d_1(x)\geq 0,\dots,d_t(x)\geq 0\}.
\]
Then, \reff{copo:equ:1} is equivalent to  the scalar polynomial optimization
	\be  \label{copo:equ:2} 
\left\{ \baray{rl}
\min & \xi^T[x]_d  \\
\st &  \tr{X(x)}=1,\\
&    f_{\min}-f(x) \geq  0,\\
&d_i(x)\geq 0~(i=1,\dots,t),\\
&  \|X(x)\|_F^2\leq 1,~\|X(x)\|_F^2\geq  \frac{1}{n}.\\
\earay \right.
\ee
Note that $w^{(k)}$ is asymptotically optimal, i.e., $\langle \xi^T[x]_d, w^{(k)}\rangle$ converges to the optimal value of \reff{copo:equ:1}. Furthermore,   every minimizer of \reff{momran:1} is also feasible for the 
moment relaxation of \reff{copo:equ:2}, by using subvectors. It follows from 	\cite[Corollary 3.5]{msch}
	that the sequence $\{u^{(k)}:=(w^{(k)}_{e_1}, \dots,w^{(k)}_{e_{\sigma(n)}})\}_{k=k_0}^{\infty}$  converges to the unique minimizer $x^*$.   The constraints $L_{X}^{(k)}[w] \succeq 0$ and $L^{(k)}_{\tr{X}-1}(w)=0$  imply that 
	\[
	X(u^{(k)})\succeq 0,\quad \tr{X(u^{(k)})}=1.
	\]
		Since  $f(x^*)\leq  f_{\min}<0$,  we know that $f(u^{(k)})<0$ when $k$ is sufficiently large.
	Therefore, for some sufficiently large $k$, Algorithm~\ref{ag:CopoSDP} returns $ u^{(k)}$.

\end{proof}

\section{An algorithm for testing $\mathcal{S}_+^n\times \mR_+^m$-copositivity}
\label{sc:othsem}

In this section, we propose an  algorithm to test $K$-copositivity of homogeneous polynomials, where $K$ is the  direct product of  the positive semidefinite cone $\mathcal{S}_+^{n}$ and the  nonnegative orthant $\mR_{+}^m$, i.e.,  $K=\mathcal{S}_+^n\times \mR_+^m$. 
Denote the variable vectors  
\[
x:=(x_{11},x_{12},\dots,x_{1n},x_{22},x_{23},\dots,x_{n-1,n},x_{nn}),\quad y:=(y_1,\dots,y_m).
\]
Let  $X(x)$
be the $n$-by-$n$ symmetric  matrix whose $(i,j)$-th entry is $x_{ij}$ for $i\leq j$. 
Consider the  problem
\be  \label{d:copo}
\left\{ \baray{rl}
\min & f(x,y)  \\
\st &\tr{X(x)}+y_1+\cdots+y_m=1,\\
& 	X(x) \succeq 0, y_1\geq 0,\dots,y_m\geq 0, \\
\earay \right.
\ee
where $f(x,y)$ is a homogeneous polynomial of degree $d$. Denote  
\[
h(x):=\tr{X(x)}-1,
\] 
and let $f^{\prime}_{\min}$ be the optimal value of $\reff{d:copo}$. Then,  $f(x,y)$ is copositive over  $ \mathcal{S}_+^{n}\times \mR_{+}^m$ if and only if $f^{\prime}_{\min}\geq 0$.


First, we show that  the nondegeneracy condition  holds at every feasible point of  \reff{d:copo}.

\begin{prop}\label{d:ndchold}
Let $(u,v)\in \mR^{\sigma(n)}\times \mR^{m}$ be a feasible point of  \reff{d:copo} with $\rank~ X(u)=r$. Then, the NDC \reff{CQ} holds at  $(u,v)$.
\end{prop}

\begin{proof}
Suppose  the zero entries of $v$ are $v_{k_1},\dots,v_{k_{\ell}}$.
Let  $\{q_1,\dots,q_{n-r}\}$ be an orthonormal basis of the kernel of  $X(u)$.	In view of  Proposition \ref{equindc}, it suffices to show that  the vectors 
\[
\begin{bmatrix} \nabla h(u)\\
e\\
\end{bmatrix},\quad  \begin{bmatrix} \nabla X(u)^*[\frac{q_i q_j^{T}+q_j q_i^{T}}{2}]\\
0\\
\end{bmatrix}~(1\leq i\leq j \leq n-r),\quad \begin{bmatrix}0\\
e_{k_t}
\end{bmatrix}~(t=1,\dots, \ell),
\]
are linearly independent. Here, $e\in \mR^m$ is the  vector with all entries equal to 1, and $e_{k_t}$ is the $k_t$-th  standard
basis vector in $\mR^m$.  Suppose there exist  real scalars $\mu$, $\lambda_{\ij}$ $(1\leq i\leq j \leq n-r)$, $\mu_{t}$ $(t=1,\dots,\ell)$ such that
\be \nonumber
\mu \begin{bmatrix} \nabla h(u)\\
	e\\
\end{bmatrix}+ \sum\limits_{1\leq i\leq j \leq n-r}\lambda_{ij}	 \begin{bmatrix} \nabla X(u)^*[\frac{q_i q_j^{T}+q_j q_i^{T}}{2}]\\
0\\
\end{bmatrix}+ \sum\limits_{t=1}^{\ell} \mu_{t} \begin{bmatrix}0\\
e_{k_t}
\end{bmatrix}=0.
\ee
This is equivalent to
\be \label{d1:ndc1}
\mu  \nabla h(u)+ \sum\limits_{1\leq i\leq j \leq n-r}\lambda_{ij}	 \nabla X(u)^*[\frac{q_i q_j^{T}+q_j q_i^{T}}{2}]=0,
\ee
\be \label{d2:ndc1}
\mu e+  \sum\limits_{t=1}^{\ell} \mu_{t}	e_{k_t}=0.
\ee 
We show that all these scalars are zero, which completes the proof.


If $\ell=m$, then $v=0$ and   $h(u)=\tr{X(u)}-1=0$. Similar to Proposition \ref{ndchold},   \reff{d1:ndc1} implies that $\mu=0$ and $\lambda_{ij}=0$ for  $1\leq i\leq j \leq n-r$. Therefore, we also have $\mu_{t}=0$ for $t=1,\dots,\ell$.

 If $\ell<m$,   equation \reff{d2:ndc1} implies  that $\mu=0$ and $\mu_{t}=0$ for $t=1,\dots,\ell$. Then,  \reff{d1:ndc1}   reduces to $\nabla X(u)^*[M]=0$ for the matrix 
\[
M=\sum\limits_{1\leq i\leq j \leq n-r}\lambda_{ij}\cdot \frac{q_i q_j^{T}+q_j q_i^{T}}{2}.
\]
We have that $M=0$, and  for  $1 \leq i_0 < j_0 \leq n-r$, it holds that
\[
\baray{ll}
0=q_{i_0}^TMq_{j_0}&=q_{i_0}^T(\sum\limits_{1 \leq i \leq j \leq n-r} \lambda_{ij}\cdot \frac{q_i q_j^{T}+q_j q_i^{T}}{2})q_{j_0}\\
&=\sum\limits_{1 \leq i \leq j \leq n-r} \lambda_{ij}\cdot \frac{q_{i_0}^Tq_i q_j^{T}q_{j_0}+q_{i_0}^Tq_j q_i^{T}q_{j_0}}{2}=\frac{1}{2}\lambda_{i_0j_0}.\\
\earay 
\]
Similarly, we have $0=q_{i_0}^TMq_{i_0}=\lambda_{i_0i_0}$ for $1\leq i_0\leq n-r$. Hence, we conclude that  $\lambda_{ij}=0$ for $1\leq i\leq j \leq n-r$.
\end{proof}

Let $(u,v)\in \mR^{\sigma(n)}\times \mR^{m}$ be a local minimizer of \reff{d:copo}.  By Proposition \ref{d:ndchold} and  Theorem \ref{KKT:cond},  there exist $\mu\in \mR,(\mu_1,\dots,\mu_m)\in \mR_+^m$ and  $\Lambda \in \mathcal{S}_+^n$ such that 
\be \label{d:kktnco}
\begin{bmatrix} \nabla_x f(u,v)\\
	\nabla_y f(u,v)\\
\end{bmatrix}=\mu \begin{bmatrix} \nabla h(u)\\
	e\\
\end{bmatrix}+ 	 \begin{bmatrix} \nabla X(u)^*[\Lambda]\\
	0\\
\end{bmatrix}+ \sum\limits_{t=1}^{m} \mu_{t} \begin{bmatrix}0\\
	e_{t}
\end{bmatrix},
\ee 
\[
\tr{\Lambda X(u)}=0,\quad \mu_1v_1=\cdots=\mu_mv_m=0. 
\]
By Euler's identity for homogeneous polynomials, we have 
\[
(u,v)^T\begin{bmatrix} \nabla_x f(u,v)\\
	\nabla_y f(u,v)\\
\end{bmatrix}=d\cdot f(u,v).
\]
Note that 
\[
(u,v)^T \begin{bmatrix} \nabla h(u)\\
	e\\
\end{bmatrix}=\tr{X(u)}+v_1+\cdots+v_m=1,
\]
\[
u^T(\nabla X(u)^* [\Lambda])=\tr{\Lambda X(u)}=0.
\]
Multiplying by $(u,v)^T$ on both sides of   \reff{d:kktnco}, we obtain
\[
\baray{ll}
d\cdot f(u,v)&=\mu \cdot (u,v)^T\begin{bmatrix} \nabla h(u)\\
	e\\
\end{bmatrix}+u^T(\nabla X(u)^* [\Lambda])+\sum\limits_{t=1}^{m} \mu_{t} v_t\\
&= \mu. \\
\earay
\]
Note that   \reff{d:kktnco} is equivalent to
\[
\partial_{x_{ii}}f(u,v)=\mu+\Lambda_{ii}~(i=1,\dots,n),
\]
\[
\partial_{x_{ij}}f(u,v)=2\Lambda_{ij}~(1\leq i<j\leq n),
\]
\[
\partial_{y_t}f(u,v)=\mu+\mu_t~(t=1,\dots,m).
\]
Then, we have
\[
\Lambda_{ii}=\partial_{x_{ii}}f(u,v)-d\cdot f(u,v)~(i=1,\dots,n),
\]
\[
\Lambda_{ij}=\frac{1}{2}\partial_{x_{ij}}f(u,v)~(1\leq i<j\leq n),
\]
\[
\mu_t=\partial_{y_t}f(u,v)-d\cdot f(u,v)~(t=1,\dots,m).
\]
Let $\Theta(x,y)$ be the $n$-by-$n$ symmetric polynomial matrix with entries given by $\Lambda_{ij}$, i.e.,
{\small
	\be \label{exp:theta}
	\Theta(x,y)= \left[\begin{matrix}
		\partial_{x_{11}}f-d\cdot f &\frac{1}{2}\partial_{x_{12}}f&\cdots &\frac{1}{2}\partial_{x_{1n}}f\\
		\frac{1}{2}\partial_{x_{12}}f&\partial_{x_{22}}f-d\cdot f&\cdots &\frac{1}{2}\partial_{x_{2n}}f\\
		\vdots 	&\vdots &\ddots &\vdots \\
		\frac{1}{2}\partial_{x_{1n}}f&\frac{1}{2}\partial_{x_{2n}}f&\cdots& \partial_{x_{nn}}f-d\cdot f
	\end{matrix}\right],
	\ee
}
 and let 
 \[
 p_t(x,y)=\partial_{y_t}f-d\cdot f \text{\quad for\,\,} t=1,\dots,m.
 \]

Since $X(u)\succeq 0$, $\Theta(u) \succeq 0$, and $\tr{\Theta(u) X(u)}=0$, we have  $X(u)\Theta(u)=0$. 
 In the following, we give an estimate on the Frobenius norm of $\|X(u)\|_F+\|v\|^2$. Let $\lambda_1,\dots,\lambda_n$ be the eigenvalues of $X(u)$. Since $X(u)\succeq 0$ and $\tr{X(u)}+v_1+\cdots+v_m=1$, it holds that
\[
\lambda_1\geq 0,\dots,\lambda_n \geq 0,~ \lambda_1+\cdots+\lambda_n+v_1+\cdots+v_m=1.
\]
Then, we have that $0\leq \lambda_i \leq 1$ and 
\[
\baray{ll}
\|X(u)\|^2_F+\|v\|^2&=\lambda_1^2+\cdots+\lambda_n^2+v_1^2+\cdots+v_m^2\\
&\leq \lambda_1+\cdots+\lambda_n+v_1+\cdots+v_m\\
&=1,\\
\earay
\]
\[
\baray{ll}
\|X(u)\|^2_F+\|v\|^2&=\frac{1}{n+m}\cdot (1^2+\cdots+1^2)(\lambda_1^2+\cdots+\lambda_n^2+v_1^2+\cdots+v_m^2)\\
&\geq \frac{1}{n+m}\cdot (\lambda_1+\cdots+\lambda_n+v_1+\cdots+v_m)^2\\
&=\frac{1}{n+m},\\
\earay
\]
where the last inequality follows from Cauchy's inequality.

Since the feasible set of \reff{d:copo} is compact, its optimal value is achievable. Hence,  \reff{d:copo} is equivalent to the following strengthened reformulation:
\be  \label{d:copo:equ}
\left\{ \baray{rl}
\min & f(x,y)  \\
\st &  \tr{X(x)}+y_1+\cdots+y_m=1,\\
&   X(x)\succeq 0,~\Theta(x,y)\succeq 0,~X(x)\Theta(x,y)=0,\\
&y_t\geq 0,~p_t(x,y)\geq 0,~p_t(x,y)y_t=0~ (t=1,\dots,m),\\
&  1\geq \|X(x)\|_F^2+\|y\|^2,~\|X(x)\|_F^2+\|y\|^2\geq  \frac{1}{n+m}.\\
\earay \right.
\ee
Clearly, the optimal value of \reff{d:copo:equ} is equal to $f^{\prime}_{\min}$.

Let $d_{0}:=\lceil \frac{d+1}{2} \rceil$. 
We apply the matrix Moment-SOS relaxations to solve \reff{d:copo:equ}. For an order $k\geq  d_0$, the $k$th order SOS relaxation of \reff{d:copo:equ} is 
	\be  \label{d:sos}
	\left\{ \baray{rll}
	\max & \gamma &\\
	\st &f-\gamma \in &\ideal{\tr{X}+\sum\limits_{t=1}^m y_t-1,X\Theta,p_1y_1,\dots,p_my_m}_{2k}\\
	& &+\qmod{X,\Theta,p_1,\dots,p_m,y_1,\dots,y_m}_{2k}\\
	&  &+\qmod{1-\|X\|_F^2-\|y\|^2,\|X\|_F^2+\|y\|^2-\frac{1}{n+m}}.
	\earay \right.
	\ee
The dual problem of \reff{d:sos} is the $k$th order moment relaxation:

\be  \label{d:mom}
\left\{ \baray{rl}
\min & \langle f, z\rangle  \\
\st & L^{(k)}_{\tr{X}+y_1+\cdots+y_m-1}(z)=0,\\
&  L_{X}^{(k)}(z) \succeq 0,~L_{\Theta}^{(k)}(z) \succeq 0,~L^{(k)}_{X \Theta}(z)=0,\\
& L_{y_t}^{(k)}(z) \succeq 0,~L_{p_t}^{(k)}(z) \succeq 0,~L^{(k)}_{p_ty_t}(z)=0~(t=1,\dots,m),\\
&L_{1-\|X\|_F^2-\|y\|^2}^{(k)}(z) \succeq 0,~L_{\|X\|_F^2+\|y\|^2-\frac{1}{n+m}}^{(k)}(z) \succeq 0,\\
&\langle 1, z\rangle=1,~M_k(z) \succeq 0,~ z \in \mathbb{R}^{\mathbb{N}_{2k}^{\sigma(n)+m} } .
\earay \right.
\ee
Let $f^{\prime}_{k,sos}$ and $f^{\prime}_{k,mom}$ denote the optimal values of \reff{d:sos} and \reff{d:mom},  respectively.

The  hierarchy  \reff{d:sos}--\reff{d:mom} shares similar properties with the hierarchy \reff{sos}--\reff{mom}. 

\begin{thm}\label{d:thm:fin}
	Suppose $f(x,y)$ is a  homogeneous polynomial of degree $d$. Then, we have:
	
	\bit
	
	\item[(i)] The relaxation \reff{d:mom} is feasible, and  its optimal value is attainable.

	\item[(ii)]  $f^{\prime}_{k,sos}= f^{\prime}_{k,mom}=f^{\prime}_{\min}$ for all $k$ sufficiently large.
	
	\eit 
\end{thm}

\vspace{.2cm}

The proof of  Theorem \ref{d:thm:fin} is similar to that of Theorem \ref{thm:fin}, and we omit it for brevity. The algorithm for testing $\mathcal{S}^n_{+}\times \mR_{+}^m$-copositivity is given below.

\begin{alg}
	\label{ag2:CopoSDP} \rm
	Testing $\mathcal{S}^n_{+}\times \mR_{+}^m$-copositivity of  homogeneous polynomials. 
	
	\begin{description}
		\item [Input]  A homogeneous polynomial $f(x,y)$ of degree $d$.
		\item [Step~0]
		Choose a generic vector $\xi \in \mathbb{R}^{\mathbb{N}_d^{\sigma(n)+m}}$. Let $k:=\lceil \frac{d+1}{2}\rceil$.
		
		\item [Step~1]
		Solve the semidefinite relaxation pair \reff{d:sos}--\reff{d:mom}. If the  optimal value $f^{\prime}_{k,mom} \geq 0$, output that $f$ is $\mathcal{S}_+^n\times \mR_+^m$-copositive  and stop; otherwise,  go to Step 2.
		
		\item[Step~2]
		Solve the semidefinite program
		
		\be  \label{d:momran}
		\left\{ \baray{rl}
		\min & \langle \xi^T[x\,\, y]_d, w\rangle  \\
		\st & L^{(k)}_{\tr{X}+y_1+\cdots+y_m-1}(w)=0,\\
		& L_{X}^{(k)}(w) \succeq 0,~L_{f^{\prime}_{k,mom}-f}^{(k)}(w) \succeq 0,\\
		& L_{y_t}^{(k)}(w) \succeq 0\,\,(t=1,\dots,m),\\
		&L_{1-\|X\|_F^2-\|y\|^2}^{(k)}(w) \succeq 0,~L_{\|X\|_F^2+\|y\|^2-\frac{1}{n+m}}^{(k)}(w) \succeq 0,\\
		&\langle 1, w\rangle=1,~M_k(w) \succeq 0,~ w \in \mathbb{R}^{\mathbb{N}_{2k}^{\sigma(n)+m} } .
		\earay \right.
		\ee
		If \reff{d:momran} is feasible, compute an optimizer $w^*$ and go to Step 3;  otherwise, let $k:=k+1$ and go to Step 1.

		\item[Step~3] Let $u=(w^*_{e_1}, \dots,w^*_{e_{\sigma(n)}}) \in \mR^{\sigma(n)}$, $v=(w^*_{e_{\sigma(n)+1}},\dots,w^*_{e_{\sigma(n)+m}}) \in  \mR^{m}$. If $f(u,v)<0$, output that $f(x,y)$ is not $\mathcal{S}_+^n\times \mR_+^m$-copositive, return $(u,v)$ and stop;  otherwise, let $k:=k+1$ and go to Step 1.
		
		\item [Output] A certificate that $f$ is $\mathcal{S}_+^n\times \mR_+^m$-copositive   or a refutation $0\neq (u,v)\in \mR^{\sigma(n)}\times \mR^{m}$  such that $X(u)\succeq 0,\, v\geq 0$ and  $f(u,v)<0$.

	\end{description}
\end{alg}

Algorithm \ref{ag2:CopoSDP} also terminates in  finitely many iterations, either providing a
 certificate of $\mathcal{S}_+^n\times \mR_+^m$-copositivity, or returning  a refutation.   The proof is similar to that of Theorem \ref{finitethm} and is omitted for cleanness.

\begin{thm}
	Suppose $f(x,y)$ is a homogeneous polynomial of degree $d$. Then,	Algorithm \ref{ag2:CopoSDP} terminates in finitely many iterations. To be more specific, we have:
	
	\bit 
	
	\item[(i)] If $f(x,y)$ is $\mathcal{S}_+^n\times \mR_+^m$-copositive, we have $f^{\prime}_{k,mom}\geq 0$ for all sufficiently large $k$. Consequently, Algorithm \ref{ag2:CopoSDP} outputs that $f$ is $\mathcal{S}_+^n\times \mR_+^m$-copositive.

	\item[(ii)] If $f(x,y)$ is not $\mathcal{S}_+^n\times \mR_+^m$-copositive,	Algorithm \ref{ag2:CopoSDP}  returns a nonzero
	vector $(u,v)\in \mR^{\sigma(n)}\times \mR^{m}$ satisfying $X(u)\succeq 0,v\geq 0$ such that $f(u,v)<0$.

	\eit 
\end{thm}

\bigskip

We  remark that the method developed in this section can be readily extended   to test $K$-copositivity, where  $K$ is the direct product of multiple positive semidefinite cones and  nonnegative orthants.

\section{Numerical examples}
\label{sec:num}

This section presents some examples of applying Algorithm \ref{ag:CopoSDP} to test $\mathcal{S}^n_{+}$-copositivity  and Algorithm \ref{ag2:CopoSDP} to test $\mathcal{S}^n_{+}\times \mR^m_+$-copositivity. The computations are implemented in MATLAB
R2024a, on a Lenovo Laptop with CPU@1.40GHz and RAM 32.0G.
The relaxations
\reff{sos}--\reff{mom}, \reff{d:sos}--\reff{d:mom} are modeled and solved using Yalmip  \cite{yalmip}, which calls the SDP solver SeDuMi \cite{sturm}. We refer to Remark \ref{rem1} for  some numerical settings. The columns labeled ``lower bound" represent the optimal value of the relaxation  \reff{mom} or \reff{d:mom},
while the columns labeled ``time" represent the computational time. For neatness, only four decimal digits are displayed for computational results.


\subsection{The case that $K=\mathcal{S}_+^n$}
We present some examples on  testing $\mathcal{S}^n_{+}$-copositivity using  Algorithm \ref{ag:CopoSDP}.

\begin{exm}\label{exc1}
	Consider the case where $n=2$, and 

\[
f_1=x_{11}^2 x_{12}+x_{11} x_{12}^2+x_{22}^3-3 x_{11} x_{12} x_{22},
\]	
\[
f_2=x_{11}^3+x_{12}^3+x_{22}^3-x_{11}^2 x_{12}-x_{11} x_{12}^2-x_{11}^2 x_{22}  -x_{11} x_{22}^2-x_{12}^2 x_{22}-x_{12} x_{22}^2+3 x_{11} x_{12} x_{22},
\]
\[
f_3=x_{11}^2 x_{12}+x_{12}^2 x_{22}+x_{22}^2 x_{11}-3 x_{11} x_{12} x_{22}.
\]
The polynomials $f_i(x_{11}^2,x_{12}^2,x_{22}^2)$ for $i=1,2,3$ correspond  to the Motzkin, Robinson, and Choi-Lam polynomials respectively, all of which are known to be nonnegative but not SOS. Consequently, $f_1,f_2,f_3$ are $\mR^3_+$-copositive. However, they are not $\mathcal{S}^2_{+}$-copositive. At the  order $k=2$, Algorithm \ref{ag:CopoSDP} returns refutations  $(0.9570,-0.2029,0.0430)$, $(0.5000,-0.5000,0.5000)$, $(0.9390,-0.2394,0.0610)$ for $f_1,f_2,f_3$, respectively.  The
computational results
are presented  in Table \ref{tab:exc1}.

\begin{table}[htbp]
	\centering
	\begin{tabular}{|c|c|c|c|c|}
		\hline
		\multirow{2}{*}{polynomials}& \multicolumn{2}{c|}{k=2}& \multicolumn{2}{c|}{k=3}\\
		\cline{2-5}  
		& lower bound &time &lower bound &time\\
		\hline
		$f_1$ &  -0.1213 &  0.1998  &  -0.1213& 0.3477\\
		\hline
		$f_2$ &-0.5000 & 0.1864 &-0.5000 &0.3150\\
		\hline
		$f_3$ &-0.1629 & 0.1931 &  -0.1629
		 &0.3237\\
		\hline
	\end{tabular}
	\caption{Computational results for Example \ref{exc1}}
	\label{tab:exc1}
\end{table}

\end{exm}

\begin{exm}\label{exc2}
		Consider the case where $n=3$, and 
	
	\[
	f_4=x_{11} x_{22}-x_{12}^2+x_{22} x_{33}-x_{23}^2,
	\]	
	\[
	f_5=x_{22}+x_{33}+10(x_{11} x_{22}-x_{12}^2),
	\]
	\[
	f_6=\det(X(x)).
	\]
	Here, $\det(X(x))$ is the determinant of $X(x)$.
	It can be observed that these polynomials are all  $\mathcal{S}^3_{+}$-copositive.  Algorithm \ref{ag:CopoSDP} confirms that $f_4$ and $f_5$ are $\mathcal{S}^3_{+}$-copositive  at the  order $k=2$, and that $f_6$ is $\mathcal{S}^3_{+}$-copositive  at  $k=3$, up to tiny round-off errors. The
	computational results
	are presented  in Table \ref{tab:exc2}.

	\begin{table}[htbp]
		\centering
		\begin{tabular}{|c|c|c|c|c|}
			\hline
			\multirow{2}{*}{polynomials}& \multicolumn{2}{c|}{k=2}& \multicolumn{2}{c|}{k=3}\\
			\cline{2-5}  
			& lower bound &time &lower bound &time\\
			\hline
			$f_4$ & -2.8175$\cdot 10^{-10}$  & 0.3497  & -1.3782 $\cdot 10^{-9}$ & 4.4723\\
			\hline
			$f_5$ & -1.4491$\cdot 10^{-14}$  & 0.5040   &  -1.2175$\cdot 10^{-14}$&4.6743 \\
			\hline
			$f_6$ &  -0.0208 &  0.3850   & -2.8439$\cdot 10^{-9}$ & 4.3190\\		
			\hline
		\end{tabular}
		\caption{Computational results for Example \ref{exc2}}
		\label{tab:exc2}
	\end{table}

\end{exm}

\begin{exm}\label{exc3}
	Consider the polynomial 
	\[
	f(x)=\sum\limits_{i=1}^{n-1} x_{ii} x_{i+1,i+1}-x_{i,i+1}^2.
	\]	
Since the term $x_{ii} x_{i+1,i+1}-x_{i,i+1}^2$ is the  determinant of the principal submatrix of $X(x)$ corresponding to  the $i$th, $(i+1)$-th rows and  columns,	we know that $f(x)$ is  $\mathcal{S}^n_{+}$-copositive.  For $n=2,3,4$, Algorithm \ref{ag:CopoSDP} confirms that $f$ is $\mathcal{S}^n_{+}$-copositive  at the  order $k=2$.  The
	computational results for these values of  $n$ at  $k=2$
	are presented  in Table \ref{tab:exc3}.

	\begin{table}[htbp]
		\centering
		\begin{tabular}{|c|c|c|c|c|}
			\hline
		  n& 2& 3 &4\\
			\hline
		    lower bound & -1.6894$\cdot10^{-10}$   &  -2.8175$\cdot10^{-10}$  &-4.4901$\cdot10^{-8}$   \\
			\hline
			time & 0.2037  & 0.3497   & 3.1727  \\
			\hline

		\end{tabular}
		\caption{Computational results for Example \ref{exc3}}
		\label{tab:exc3}
	\end{table}

\end{exm}

\begin{exm}\label{exc11}
Consider the case where $n=4$, 
	$f(x)=\operatorname{tr}((XA)^2)$, and 
	$$
	A=\begin{bmatrix}
		1 & -0.72 & -0.59 & 1 \\
		-0.72 & 1 & -0.6 & -0.46 \\
		-0.59 & -0.6 & 1 & -0.6 \\
		1 & -0.46 & -0.6 & 1
	\end{bmatrix}.
	$$	
This example is from \cite{fgnr24}, where the authors  guess that $f$ is  $\mathcal{S}_{+}^4$-copositive. At the  order $k=3$, Algorithm \ref{ag:CopoSDP}  confirms that $f$ is $\mathcal{S}^4_{+}$-copositive, up to tiny round-off errors. The
computational results 
are presented  in Table \ref{tab:exc11}.

\begin{table}[htbp]
	\centering
	\begin{tabular}{|c|c|c|c|}
		\hline
		 \multicolumn{2}{|c|}{k=2}& \multicolumn{2}{c|}{k=3}\\
		\cline{1-4}  
		 lower bound &time &lower bound &time\\
		\hline
		  $-3.2167\cdot 10^{-5}$ &  4.8262   &  $ -1.2428\cdot 10^{-7}$& 1109.5331\\
		\hline
	\end{tabular}
	\caption{Computational results for Example \ref{exc11}}
	\label{tab:exc11}
\end{table}

\end{exm}

\begin{exm}\label{exc5}
	Consider the case where $n=5$, 
	$f(x)=\operatorname{tr}((X^2A+XAX)$, and 
	$A$ is the  Horn matrix, i.e., 
	$$
	A=\begin{bmatrix}
		1 & -1 & 1 & 1 & -1 \\
		-1 & 1 & -1 & 1 & 1 \\
		1 & -1 & 1 & -1 & 1 \\
		1 & 1 & -1 & 1 & -1 \\
		-1 & 1 & 1 & -1 & 1
	\end{bmatrix}.
	$$
	This example is from \cite{fgnr24}, where it is verified that $f$ is not  $\mathcal{S}_{+}^5$-copositive. At the  order $k=2$, Algorithm \ref{ag:CopoSDP} 
	 returns a refutation
	\[
	X^*=	\begin{bmatrix}
		0.2000 & 0.0618 & -0.1618 & -0.1618 & 0.0618 \\
		0.0618 & 0.2000 & 0.0618 & -0.1618 & -0.1618 \\
		-0.1618 & 0.0618 & 0.2000 & 0.0618 & -0.1618 \\
		-0.1618 & -0.1618 & 0.0618 & 0.2000 & 0.0618 \\
		0.0618 & -0.1618 & -0.1618 & 0.0618 & 0.2000
	\end{bmatrix}.
	\]
	The computation took around 70.1866 seconds.

\end{exm}

\subsection{The case that $K=\mathcal{S}_+^n\times \mR_+^m$}
We present some examples on  testing $\mathcal{S}^n_{+}\times \mR_+^m$-copositivity using  Algorithm \ref{ag2:CopoSDP}.

\begin{exm}\label{aexc1}
	Consider the case where  $n=2$, $m=2$, $f(x,y)=(x^T ~y^T)A(x ~y)$, and $A$ is the  Horn matrix as in Example \ref{exc5}.
	  At the  order $k=3$, Algorithm \ref{ag2:CopoSDP} 
	  returns a refutation $(u,v)$, where 
	  \[
	  u=(0.0769,   -0.0769,    0.0769),\,\, v=(0.4231,    0.4231).
	  \]
	  The computation took around 4.9563 seconds.

	\end{exm}

\begin{exm}\label{aexc1:2}

	Consider the case where  $n=3$, $m=1$,  $f(x,y)=(x^T ~y^T)A(x ~y)$, and $A$ is the Hoffman-Pereira matrix, i.e.,
$$
A=\begin{bmatrix}
	1 & -1 & 1 & 0 & 0 & 1 & -1 \\
	-1 & 1 & -1 & 1 & 0 & 0 & 1 \\
	1 & -1 & 1 & -1 & 1 & 0 & 0 \\
	0 & 1 & -1 & 1 & -1 & 1 & 0 \\
	0 & 0 & 1 & -1 & 1 & -1 & 1 \\
	1 & 0 & 0 & 1 & -1 & 1 & -1 \\
	-1 & 1 & 0 & 0 & 1 & -1 & 1
\end{bmatrix}.
$$
  At the  order $k=2$, Algorithm \ref{ag2:CopoSDP}  
   returns a refutation $(u,v)$, where 
  \[
  u=(0.4172 ,  -0.0716,   -0.2367,    0.0123,    0.0406 ,   0.1342),\,\, v= 0.4362.
  \]
   The computation took around 1.3549 seconds.

\end{exm}


\begin{exm}\label{aexc2}
	
	Consider the case where $n=2$, $m=4$,  $f(x,y)=(x^T ~y^T)A_{\alpha}(x ~y)$, and $A_{\alpha}$ is a perturbation of  the Hoffman-Pereira matrix, i.e.,
	$$
	A_{\alpha}=\begin{bmatrix}
		1 & -1 & 1 & 0 & 0 & 1 & -1 \\
		-1 & 1 & -1 & 1 & 0 & 0 & 1 \\
		1 & -1 & 1 & -1 & 1 & 0 & 0 \\
		0 & 1 & -1 & (1+\alpha)^2  & -1 & 1 & 0 \\
		0 & 0 & 1 & -1 & (1+\alpha)^2 & -1 & 1 \\
		1 & 0 & 0 & 1 & -1 & (1+\alpha)^2  & -1 \\
		-1 & 1 & 0 & 0 & 1 & -1 & (1+\alpha)^2 
	\end{bmatrix}.
	$$
 For different values of $\alpha$, Algorithm \ref{ag2:CopoSDP}  either certifies that $f$ is $\mathcal{S}^2_{+}\times \mR_+^4$-copositive or returns a refutation at the  order $k=2$. The
 computational results for the order $k=2$
 are presented  in Table \ref{tab:aexc2}.

	\begin{table}[htbp]
		\centering
		\begin{tabular}{|c|c|c|c|}
			\hline
			$\alpha$&lower bound &time &copositivity\\
			\cline{2-3}  
			\hline
				0.01 &    -0.0229 &  1.6916    & No  \\
			\hline
				0.02 & -0.0152  & 1.5458   & No  \\
			\hline
				0.03 &  -0.0075 &  1.6280   & No  \\
			\hline
				0.04 &   1.2918$\cdot 10^{-4}$ &   1.5394 & Yes   \\
			\hline
				0.05 &  0.0078&  1.5813  &  Yes  \\
			\hline
				0.06 & 0.0154  &  1.7282   & Yes   \\
			\hline
				0.07 & 0.0229  & 1.5787      &  Yes  \\
			\hline
				0.08 & 0.0304  & 1.5329      &  Yes \\
			\hline
				0.09 &   0.0379  &  1.4894    &   Yes \\
			\hline
				0.10 & 0.0453  & 1.5831   &   Yes \\
			\hline
		\end{tabular}
		\caption{Computational results for Example \ref{aexc2}}
		\label{tab:aexc2}
	\end{table}

\end{exm}

\section{Conclusions and discussions}
\label{sec:dis}

In this paper, we  propose  an efficient   algorithm   for testing  $\mathcal{
S}_{+}^n$-copositivity of  homogeneous polynomials. It involves solving a sequence of semidefinite programs. A remarkable property of the algorithm is that it always terminates in finitely many iterations, either certifying the  copositivity or returning a   vector that refutes copositivity.
  We further generalize this algorithm to test $\mathcal{S}_+^n\times \mR_+^m$-copositivity  of homogeneous polynomials.  
Preliminary numerical experiments
demonstrate the efficiency of our algorithms. Moreover, the methods developed in this paper can be naturally extended to test   $K$-copositivity when $K$ is the direct product of multiple positive semidefinite cones and  nonnegative orthants.

The algorithms rely on the strengthened  matrix optimization reformulations,  inspired by the work  \cite{hnlme}.
In  \cite{hnlme}, the authors introduced tight Moment--SOS relaxations for polynomial matrix optimization without equality constraints.
We remark that the methods and proof framework developed in this paper can be extended to construct tight relaxations for polynomial matrix optimization  with equality constraints,  extending the approach  in \cite{hnlme}.

Consider the optimization:
\be  \label{con:nsdp}
\left\{ \baray{rl}
\min\limits_{x\in \mR^n} & f(x)  \\
\st & h_1(x)=0,\dots,h_{\ell}(x)=0,\\
& 	G_1(x)\succeq 0,\dots,G_s(x)\succeq 0, \\
\earay \right.
\ee
where $f(x),h_1(x),\dots,h_{\ell}(x)\in \mR[x]$, and each $G_t(x)$ is an $m_t\times m_t$ symmetric polynomial matrix.  When the  optimal value $f^*_{\min}$ of \reff{con:nsdp} is achievable at a critical point,   the first order optimality conditions \reff{kkt} can be imposed as constraints. Consequently, \reff{con:nsdp} is equivalent to 
\be  \label{sec3:nsdp:eq}
\left\{ \baray{rl}
\min & f(x)  \\
\st & h_1(x)=0,\dots,h_{\ell}(x)=0,\\
& \nabla f(x)-\sum\limits_{i=1}^{\ell} \mu_i \nabla h_{i}(x)-\sum\limits_{t=1}^{s}\nabla G_t(x)^*[\Lambda_t] = 0,\\
& 	G_1(x)\succeq 0,\dots,G_s(x)\succeq 0, \\
& 	\Lambda_1\succeq 0,\dots,\Lambda_s\succeq 0, \\
&  G_1(x)\Lambda_1=\cdots= G_s(x)\Lambda_s=0,\\
& \mu_1,\dots,\mu_{\ell}\in \mR,~ x\in \mR^n,~  \Lambda_1\in \mathcal{S}^{m_1},\dots, \Lambda_s\in \mathcal{S}^{m_s}.
\earay \right.
\ee
The above equality  constraints can be expressed as linear equations in the multiplier variables $\mu_1,\dots,\mu_{\ell},\Lambda_1,\cdots,\Lambda_s$, with  coefficients being polynomials in $x$. Hence, there exists a polynomial matrix $P(x)$ such that 
\begin{equation} \nonumber
	P(x)\cdot \left[\begin{array}{c}
		\mu_1 \\
		\vdots\\
		\mu_{\ell}\\
		\uvec{\Lambda_1} \\
		\vdots\\
		\uvec{\Lambda_s}\\
	\end{array}\right]:=\left[\begin{array}{c}
		\sum\limits_{i=1}^{\ell} \mu_i \nabla h_{i}(x)+\sum\limits_{t=1}^{s}\nabla G_t(x)^*[\Lambda_t]  \\
		\mu_1h_{1}(x)\\
		\vdots\\
		\mu_{\ell}h_{\ell}(x)\\
		\vec{G_1(x)\Lambda_1} \\
		\vdots\\
		\vec{G_s(x)\Lambda_s}\\
	\end{array}\right] =\left[\begin{array}{c}
		\nabla f(x) \\
		0 \\
		\vdots\\
		0\\
	\end{array}\right],
\end{equation}
where $\uvec{\Lambda_t}$ is the  vectorization of the upper triangular entries of $\Lambda_t$, i.e., 
\[
\uvec{\Lambda_t}=\left[\begin{array}{lllllll}
	(\Lambda_t)_{11} & \cdots & (\Lambda_t)_{1 m_t} &  (\Lambda_t)_{22} & \cdots & (\Lambda_t)_{m_t-1, m_t} & (\Lambda_t)_{m_t,m_t}
\end{array}\right]^T,
\]
and  $\vec{G_t\Lambda_t}$  is the full vectorization of the  matrix $G_t\Lambda_t$, i.e., 
{\small
\[
\vec{G_t\Lambda_t}=\left[\begin{array}{lllllll}
	(G_t\Lambda_t)_{11} & \cdots & (G_t\Lambda_t)_{1 m_t} &   \cdots &(G_t\Lambda_t)_{m_t,1} & \cdots & (G_t\Lambda_t)_{m_t,m_t}
\end{array}\right]^T,
\]
}

If the matrix $P(x)$ is nonsingular (i.e., $ P(x)$ has full column rank for all $x\in \mathbb{C}^{n}$),  it follows from   \cite[Proposition 5.2]{Tight18} that there exists  a   polynomial matrix $L(x)$ such that
\be \nonumber
L(x)P(x)=I_{\ell+\sigma(m_1)+\cdots+\sigma(m_s)}.
\ee
This implies  that
\be \nonumber
\left[\begin{array}{c}
	\mu_1 \\
	\vdots\\
	\mu_{\ell}\\
	\uvec{\Lambda_1} \\
	\vdots\\
	\uvec{\Lambda_s}\\
\end{array}\right]=L(x)\left[\begin{array}{c}
\nabla f(x) \\
0 \\
\vdots\\
0\\
\end{array}\right],
\ee 
which provides polynomial expressions for the multiplier variables,   denoted by $p_1(x),\dots,p_{\ell}(x)$, $\Theta_1(x),\cdots,\Theta_s(x)$. 
Hence,  \reff{con:nsdp} is also equivalent to 
\be  \label{nsdp:equ}
\left\{ \baray{rl}
\min & f(x)  \\
\st & h_1(x)=0,\dots,h_{\ell}(x)=0,\\
& \nabla f(x)-\sum\limits_{i=1}^{\ell} p_i(x) \nabla h_{i}(x)-\sum\limits_{t=1}^{s}\nabla G_t(x)^*[\Theta_t(x)] = 0,\\
& 	G_1(x)\succeq 0,\dots,G_s(x)\succeq 0, \\
& 	\Theta_1(x)\succeq 0,\dots,\Theta_s(x)\succeq 0, \\
&  G_1(x)\Theta_1(x)=\cdots= G_s(x)\Theta_s(x)=0.
\earay \right.
\ee

The matrix Moment-SOS relaxations can be applied to solve \reff{nsdp:equ}. Denote the polynomial matrix tuples
\begin{equation}\nonumber
	\begin{split}
		\Phi:=&\{h_1,\dots,h_{\ell}\}\cup
		\{\nabla f-\sum\limits_{i=1}^{\ell} p_i \nabla h_{i}-\sum\limits_{t=1}^{s}\nabla G_t^*[\Theta_t] \}
		\cup \{G_1\Theta_1,\dots, G_s\Theta_s\},
		\\
		\Psi:= & \left\{G_1,\dots, G_s\right\}\cup \left\{\Theta_1,\dots, \Theta_s\right\}.
	\end{split}
\end{equation}
For an order $k$, the $k$th order SOS relaxation  of \reff{nsdp:equ} is 

\be  \label{con:sos}
\left\{ \baray{rl}
\max & \gamma \\
\st &f-\gamma \in \ideal{\Phi}_{2k}+\qmod{\Psi}_{2k}.
\earay \right.
\ee
The dual problem of \reff{con:sos} is the $k$th order moment relaxation:

\be  \label{con:mom}
\left\{ \baray{rl}
\min & \langle f, z\rangle  \\
\st &~L^{(k)}_{\phi}(z)=0~(\phi \in \Phi),\\
& ~ L_{\psi}^{(k)}(z) \succeq 0~(\psi \in \Psi),\\
&\langle 1, z\rangle=1,~~M_k(z) \succeq 0,~z\in \mathbb{R}^{\mathbb{N}_{2k}^{n}}.\\
\earay \right.
\ee
Let $f^{*}_{k,sos}$ and $f^{*}_{k,mom}$ denote the optimal values of \reff{con:sos} and \reff{con:mom},  respectively.

Using similar proof techniques as in Theorem \ref{thm:fin}, we can show that the hierarchy \reff{con:sos}--\reff{con:mom} has finite convergence.

\begin{thm}
	Suppose that the  optimal value  of \reff{con:nsdp} is achievable at a critical point,    the matrix $P(x)$ is nonsingular and the set $\ideal{\Phi}+\qmod{\Psi}$ is Archimedean. Then, we have $f^{*}_{k,sos}=f^{*}_{k,mom}=f^*_{\min}$ for all  $k$ sufficiently large.
	
\end{thm}

\hspace{.5em}


\bigskip
\noindent
{\bf Acknowledgements.}
The authors are deeply grateful to Prof. Samuel Burer for his motivating insights and valuable suggestions on this paper.

\end{document}